\documentclass[11pt,reqno]{amsart}
\usepackage{a4wide}
\usepackage{amsthm, amssymb, amstext, amscd, amsfonts, amsxtra, latexsym, amsmath, comment, tikz,
	mathrsfs, esint, stmaryrd, cancel}
\usepackage{color}
\usepackage{fancybox}
\usepackage{fullpage}
\usepackage{hyperref}
\usepackage[english]{babel}
\usepackage[latin1]{inputenc}
\usepackage{enumitem}
\usepackage{mdframed}

\usepackage{graphicx} 

\numberwithin{equation}{section}
\newtheorem{theorem}{Theorem}
\numberwithin{theorem}{section}
\newtheorem{lemma}[theorem]{Lemma}
\newtheorem*{theorem*}{Theorem}

\newtheorem{proposition}[theorem]{Proposition}
\newtheorem{corollary}[theorem]{Corollary}

\newtheorem*{question*}{Question}

\theoremstyle{remark}
\newtheorem*{remark}{Remark}

\theoremstyle{definition}

\newenvironment{psmallmatrix}{\left(\begin{smallmatrix}}{\end{smallmatrix}\right)}

\newcommand{\Z}{\mathbb{Z}}
\newcommand{\N}{\mathbb{N}}

\newcommand{\R}{\mathbb{R}}
\newcommand{\C}{\mathbb{C}}

\newcommand{\Span}{\operatorname{span}}

\newcommand{\im}{\operatorname{Im}}
\newcommand{\sgn}{\operatorname{sgn}}

\renewcommand{\H}{\mathbb{H}}
\newcommand{\SL}{\text{\rm SL}}

\renewenvironment{proof}[1][Proof]{\begin{trivlist} \item[\hskip \labelsep {\bfseries #1:}]}{\qed\end{trivlist}}

\title{Indefinite theta functions arising in Gromov-Witten Theory of elliptic orbifolds}
\author{Kathrin Bringmann, Jonas Kaszian, and Larry Rolen}

\address{Mathematical Institute, University of Cologne, Weyertal 86-90, 50931 Cologne, Germany}
\email{kbringma@math.uni-koeln.de}
\email{jkaszian@math.uni-koeln.de}

\address{Hamilton Mathematics Institute \& School of Mathematics, Trinity College, Dublin 2, Ireland}
\email{lrolen@maths.tcd.ie}

\thanks{
The research of the first author is supported by the Alfried Krupp Prize for Young University
Teachers  of  the  Krupp  foundation and the
Collaborative Research Centre / Transregio on Symplectic Structures in Geometry, Algebra and Dynamics (CRC/TRR 191) of the German Research Foundation. }

\begin{document}
\maketitle
\begin{abstract}
In this paper, we consider natural geometric objects coming from Lagrangian Floer theory and mirror symmetry. Lau and Zhou showed that some of the explicit Gromov-Witten potentials computed by Cho, Hong, Kim, and Lau are essentially classical modular forms. Recent work by Zwegers and two of the authors determined modularity properties of several simpler pieces of the last, and most mysterious, function by developing several identities between functions with properties generalizing those of the mock modular forms in Zwegers' thesis. Here, we complete the analysis of all pieces of Cho, Hong, Kim, and Lau's functions, inspired by recent work of Alexandrov, Banerjee, Manschot, and Pioline on similar functions. Combined with the work of Lau and Zhou, as well as the aforementioned work of Zwegers and two of the authors, this affords a complete understanding of the modularity transformation properties of the open Gromov-Witten potentials of elliptic orbifolds of the form $\mathbb P^1_{a,b,c}$ computed by Cho, Hong, Kim, and Lau. It is hoped that this will provide a fuller picture of the mirror-symmetric properties of these orbifolds in subsequent works.
\end{abstract}

\section{Intorduction and statement of results}

Cho, Hong, and Lau \cite{ChoHongLau} described open Gromov-Witten potentials for elliptic orbifolds (and homological mirror symmetry). Explicit expressions for these were computed by Cho, Hong, Kim, and Lau \cite{ChoHongKimLau}.
Recently, Lau and Zhou \cite{LauZhou} investigated the modularity properties of some of these Gromov-Witten potentials. In the course of their work, they showed that several of them are essentially modular forms, a fact which they show closely related to their mirror-symmetric interpretation. More precisely, they considered the four elliptic $\mathbb P^1$ orbifolds denoted by $\mathbb P^1_a$ for $a\in\{(3,3,3),(2,4,4),(2,3,6),(2,2,2,2)\}$. For these choices of $a$, Cho, Hong, Kim, and Lau explicitly computed the open Gromov-Witten potentials $W_q(X,Y,Z)$ of $\mathbb P^1_a$, which are polynomials in the variables $X,Y,Z$. The reader is also referred to \cite{ChoHongKimLau,ChoHongLau2,ChoHongLee} for related results, as well as to Sections 2 and 3 of \cite{LauZhou} for the definitions of the relevant geometric objects.
These generating functions turn out to have quite natural modularity properties, as Lau and Zhou proved in Theorem 1.1 of \cite{LauZhou}.
\begin{theorem*}[Lau, Zhou]\label{LauZhouTheorem}
	For $a\in\{(3,3,3),(2,4,4),(2,2,2,2)\}$, the coefficients of
	$W_q(X,Y,Z)$ are essentially linear combinations of modular forms.
\end{theorem*}

Such  results are particularly useful as they allow one to extend the domain of these potentials to global moduli spaces.  In fact, this connection provides the geometric intuition for why modular or at least near-modular, behavior may be expected (cf. \cite{ChoHongKimLau}). In the last case $a=(2,3,6)$, the relevant functions fail to be combinations of ordinary modular forms. However, it is natural to ask whether a suitably modified transformation still holds.
\begin{question*}[Lau, Zhou]\label{LauZhouQuestion}
	Can a simple description of the modular transformations of $W_q(X,Y,Z)$ be given for $a=(2,3,6)$?
\end{question*}
The first steps towards addressing Lau and Zhou's question were taken in \cite{BRZ},  where new generalizations of mock modular forms were defined and utilized. However, several remaining pieces remained out of reach due to lack of theoretical structure for such series. In the meantime, important new work of Alexandrov, Banerjee, Manschot, and Pioline \cite{ABMP} has further explained and generalized the class of functions considered in \cite{BRZ}, leading to a new theory of non-holomorphic modular objects extending those considered in Zwegers' seminal thesis \cite{Zw}. The understanding of the structure of these non-holomorphic modular forms was furthered by Kudla in \cite{Kudla}, where he showed that they can be viewed as integrals of Kudla-Millson theta series (cf. \cite{KudlaMillson}). Further extensions to a general, geometric setting were given by Westerholt-Raum in \cite{WR}.  The authors have also been informed that forthcoming work of Zagier and Zwegers will further fill in details of the general picture.

Here, we push things one step further by showing how to interpret the last pieces of Lau and Zhou's functions in terms of further classes of modular-type objects. As we shall see, these cannot be accounted for by the means presented in \cite{ABMP} due to unique features naturally arising in these functions which seem to deviate from the most basic higher type indefinite theta functions. In particular, we solve this question by providing ``simpler'' completion terms which combine with the coefficients of $W_q$ to yield modular objects. These are, in the language of Zagier and Zwegers, known as {\it higher depth mock modular forms}, which are automorphic functions characterized and inductively defined by the key property that their images under ``lowering operators'' essentially lie in lower depth spaces (for example, in the case of classical ``depth $1$'' mock modular forms, the Maass lowering operator essentially yields a classical modular form).

\begin{theorem}\label{mainthm}
	The function $c_Z$ is a higher depth mock modular form.
\end{theorem}
\begin{remark}
The higher depth structure of $c_Z$ may be deduced from its ``shadow'', the computation of which is discussed in the proof of Lemma \ref{LemmaF2F3ModCompletions} (cf. the remark after Lemma \ref{lowering}; the word ``shadow'' is justified since it is essentially the image under the Brunier-Funke operator $\xi_k$ for classical harmonic Maass forms).
\end{remark}

The answer to Lau and Zhou's original question about the modularity of $c_Z$ can be directly read off of the transformation of the completed function in Theorem \ref{mainthm}. Although we do not explicitly write it down here, the interested reader can see \eqref{cz} and the surrounding text for a discussion of how to determine it. After using explicit representations due to Lau and Zhou, the key step in the proof of Theorem \ref{mainthm} is to understand how to complete a certain indefinite theta function of signature $(3,1)$ (see \eqref{VT}
below). (Throughout, the second component denotes the number of negative eigenvalues).
The paper is organized as follows. In Section 2, we recall general indefinite theta series due to Vign\'eras and in particular give examples in signatures $(1, n)$ and $(2, n)$. In Section 3, we introduce the generalized error integrals which are our building blocks and investigate some of their properties. In Section 4, we rewrite certain generating functions in Gromov-Witten theory and start the investigation of their modularity properties. Section 5 is then devoted to modularity properties of a certain indefinite theta function of signature $(3, 1)$.

\section*{Acknowledgements}
The authors thank Jan Manschot, Boris Pioline, and Jie Zhou for useful communications and Sander Zwegers for discussions on related topics, as well as the anonymous referee for helpful suggestions on the historical discussion in the introduction.

\section{Indefinite theta functions}
\subsection{Results of Vignéras}
Let $B(n,m):=n^TAm$ be a symmetric non-degenerate bilinear form on $\R^N (N\in\N)$ which takes integral values on a lattice $L\subset \R^N$ and set $Q(n):=\frac12B(n,n)$. Further let $\mu\in L'/L$ (where $L'$ is the dual lattice of $L$), $\lambda\in\Z$, and a function $p:\R^N\rightarrow \C$.  Following Vign\'eras, we define the following indefinite theta function ($\tau= u+iv\in\H$, $z=x+iy\in\C^N$, $q:=e^{2\pi i\tau}$)
\begin{align}
\Theta_{\mu,L,A,p,\lambda}(z; \tau):=\Theta_\mu(z;\tau):=
v^{-\frac{\lambda}{2}}
\sum_{n\in \mu+L} p\left(\sqrt{v}\left(n+\frac{y}{v}\right)\right)q^{\frac12 n^T An}e^{2\pi i B(z,n)}.\label{VT}
\end{align}
Vign\'eras \cite{Vi} gave conditions under which the indefinite theta series are in fact modular.
\begin{theorem}[Vign\'eras]\label{la:Vigneras}
 Assuming the notation above, suppose that $p$ satisfies the following conditions:
	\begin{enumerate}[leftmargin=*]
		\item
		For any differential operator $D$ of order $2$ and only polynomial $R$ of degree at most $2$, $D(w)(p(w)e^{{\pi}Q(w)})$ and $R(w)p(w)e^{{\pi}Q(w)}$ belong to $\in L^2(\R^N)\cap L^1(\R^N)$.
		
		\item
	 Defining the Euler and Laplace operators ($w:=(w_1,\dots,w_N)^T,\ \partial_w:=(\frac{\partial}{\partial w_1},\dots \frac{\partial}{\partial w_N})^T$) 
		\[
		\mathcal{E}:= w^T\partial_{w}\qquad\text{and}\qquad
		\Delta=\Delta_{A^{-1}}:=\partial_w^T A^{-1}\partial_w
		,
		\]
	for some $\lambda\in\Z$ the Vign\'eras differential equation holds:
		$$
		\left(\mathcal{E}-\frac{1}{4\pi}\Delta\right)p=\lambda p.
		$$
	\end{enumerate}
	Then, assuming that $\Theta_{\mu}$ is absolutely locally convergent, we have the following modular transformations:
	\begin{align*}
	\Theta_{\mu} \left(\frac{z}{\tau};-\frac{1}{\tau}\right)
	&=\frac{\left(-i\tau\right)^{\lambda+\frac{N}{2}}}{\sqrt{\lvert L'/L\rvert}}
	e^{\frac{\pi i}{2}B\left(A^{-1}A^*,A^{-1}A^*\right)}
	\sum_{\nu\in L'/L} e^{-2\pi i B(\mu,\nu)+\frac{2\pi i}{\tau}Q(z)} \Theta_{\nu} (z;\tau),
	\\
	\Theta_{\mu}(z;\tau+1)&=
	e^{\pi i B\left(\mu+\frac12 A^{-1}A^*,\mu+\frac12 A^{-1}A^*\right)}\Theta_{\mu}(z;\tau),
	\end{align*}
	where $A^*:=(A_{1,1},\dots, A_{N,N})^T$.
\end{theorem}

To simplify the calculations below, the following lemma allows us to restrict to specific diagonal matrices. In particular, writing $A=P^{-T}DP^{-1}$ with $P\in\operatorname{GL}_N(\R)$ and $D:=\operatorname{diag}(1,\dots, 1, -1,\dots, -1)$ (with uniquely determined signs), we easily obtain the following.

\begin{lemma}\label{lemma2.2}
Assume the notation above.  If $\widetilde{p}(x):=p(Px)$ satisfies Vigneras' differential equation for $D$, then $\Theta_{\mu,L,A,p,\lambda}$ transforms like a vector-valued Jacobi form.
\end{lemma}

\begin{remark}
We frequently make use of the well-known fact that specializing the elliptic variable of Jacobi forms to torsion points yields modular forms or related objects. (See \cite{EZ} for the classical one-dimensional case).
\end{remark}

We next introduce a differential operator which, when applied to Vign\'eras' theta functions, often makes them simpler. Let
\[
X_-:=-2iv^2 \frac{\partial}{\partial\overline{\tau}}-2iv\sum_{j=1}^N y_j \partial_{\overline{z_j}},
\]
be the (multivariable) Maass lowering operator which decreases the weight of a (non-holomorphic) Jacobi form by $2$. A direct calculation gives.
\begin{lemma}\label{lowering}
We have
\[
X_-\left(\Theta_{\mu,L,A,p,\lambda}\right)=\Theta_{\mu,L,A,p_X,\lambda}
\]
with
\[
p_X(x):=\sum_{j=1}^N x_j \partial_{x_j}\left(p(x)\right).
\]

\end{lemma}

\begin{remark}
We let $\Theta_{\mu, L, A, p, \lambda}$ be the holomorphic part of $\Theta$ (whenever this is well-defined as we comment on later) and call it \emph{higher depth Jacobi form} with \emph{shadow}
$\Theta_{\mu, L, A, p_X, \lambda}$. Specializing to torsion points yields \emph{higher depth mock modular forms}.
\end{remark}

\subsection{Examples of indefinite theta functions}\label{ex.indef.theta}
 Although Vign\'eras' beautiful theorem has a simple statement, it is far from obvious how one can find appropriate functions $p$ such that the corresponding indefinite theta function converges and which has a fixed, desired ``holomorphic part''.  In his celebrated thesis \cite{Zw},  Zwegers succeeded in doing this for quadratic forms of signature $(n,1)$. In this case, the usual error function
$$
E(w):=2\int_{0}^{w} e^{-\pi t^2}dt
$$
 plays a vital role. For comparison with functions we shall need later, note that, as $w\to\pm\infty$,
\begin{equation}\label{Eas}
E(w)\sim \sgn(w).
\end{equation}
Moreover, we clearly find that
$$
E'(w)=2e^{-\pi w^2}.
$$
Also note that $E$ may be written as
$$
E(w)=\int_{\R}\sgn(t)e^{-\pi(t-w)^2}dt.
$$
\indent To discuss Zwegers' breakthrough, we now fix a quadratic form $Q$ of signature $(n,1)$. We must first discuss a few preliminary geometric considerations to describe the full behavior. The set of vectors $c\in\R^N$ with $Q(c)<0$ splits into two connected components. Two given vectors $c_1$ and $c_2$ lie in the same component if and only if $B(c_1,c_2)<0$. We fix one of the components and denote it $C_Q$. Picking any vector $c_0\in C_Q$,  we then have
\[
C_Q=\left\{c\in\R^N : Q(c)<0,\ B(c,c_0)<0\right\}
.
\]
Then the cusps are those vectors in the following set:
\[
S_Q:=
\left
\{
c=(c_1,c_2,\ldots,c_N)\in\Z^N : \operatorname{gcd}(c_1,c_2,\ldots,c_N)=1,\ Q(c)=0,\ B(c,c_0)<0
\right
\}
.
\]
A compactification of $C_Q$ may be formed by taking the union $\overline{C}_Q:=C_Q\cup S_Q$.
Defining for any $c\in \overline{C}_Q$
\[
\mathcal R(c)
:=
\begin{cases}
\R^N & \mathrm{ if }\: c\in C_Q,\\
\left\{
a\in\R^N :  B(c,a)\not\in\Z
\right\}
& \mathrm{ if }\:
c\in S_Q,
\end{cases}
\]
we set

\[
D(c)
:=
\left
\{
(z,\tau)\in\C^N\times\mathbb H : \frac{y}{v}\in\ \mathcal R(c)
\right
\}
.
\]
\indent Zwegers' indefinite theta functions, which transform as modular forms and which are (almost always) non-holomorphic, are defined as follows.
For $(z,\tau)\in D(c_1)\cap D(c_2)$, we consider the theta function
\begin{equation}\label{Ztheta}
	\theta(z;\tau):=
	\sum_{n\in\Z^N}\rho\left(n+\frac yv;\tau\right)q^{Q(n)}e^{2\pi iB(z,n)},\qquad\text{where}
\end{equation}
\begin{equation}\label{RhoDefn}
\rho(n;\tau)=\rho_Q^{c_1,c_2}(n;\tau):=\rho^{c_1}(n;\tau)-\rho^{c_2}(n;\tau)\qquad\text{with}
\end{equation}
\[
\rho^c(n;\tau)
:=
\begin{cases}
E\left(\frac{B(c,n)v^{\frac12}}{\sqrt{-Q(c)}}\right) & \mathrm{ if }\  c\in C_Q
,
\\
\sgn(B(c,n)) & \mathrm{ if }\  c\in S_Q.

\end{cases}
\]
Here and throughout we use the usual convention that for $x\in\R$, $\sgn(0):=0$ and $\sgn(x)=x/|x|$ for $x\in \R\setminus\{0\}$.
Note that the cuspidal case, $c\in S_Q$, may be viewed as a limiting case of the general situation (for example by (\ref{Eas})).

Zwegers showed that (\ref{Ztheta}) indeed converges. This is far from obvious, since the indefiniteness of $Q$ implies that $q^{Q(n)}$ is unbounded for $n\in\Z^N$. In fact, as in our case, this is one of the more subtle and substantive aspects of his proof of modularity.
The main reason for the interest in this theta function lies in the Jacobi transformation properties of $\theta$, which are described using the following auxiliary set:
\begin{equation*}
\mathcal D(c)
:=\left\{(a\tau+b,\tau) : \tau\in\mathbb H, a,b\in\R^r, B(c,a), B(c,b)\not\in\Z \right\}.
\end{equation*}
\begin{theorem}[Zwegers]\label{JacobiTransIndefTheta}
	Assuming the notation above, the function $\theta$ satisfies the following transformations:
	\begin{enumerate}[leftmargin=*]
		\item
		For all $\lambda\in\Z^N$ and $\mu\in A^{-1}\Z^N$, we have $\left(e(x):=e^{2\pi ix}\right)$
		\[
		\theta(z+\lambda\tau+\mu;\tau)=q^{-Q(\lambda)}e(-B(z,\lambda))\theta(z;\tau)
		.
		\]
		\item
		We have
		\[
		\theta(z;\tau+1)=\theta\left(z+\frac12A^{-1}A^*;\tau\right)
		.
		\]
		\item If $(z,\tau)\in \mathcal D(c_1)\cap \mathcal D(c_2)$, then
		\[
		\theta\left(\frac z{\tau};-\frac1{\tau}\right)
		=
		\frac{i(-i\tau)^{\frac N2}}{\sqrt{-\det(A)}}\sum_{n\in A^{-1}\Z^N/\Z^N}e\left(\frac{Q(z+n\tau)}{\tau}\right)\theta(z+n\tau;\tau)
		.
		\]
	\end{enumerate}
\end{theorem}
In a pathbreaking paper, Alexandrov, Banerjee, Manschot, and Pioline \cite{ABMP} then generalized Zwegers' construction to quadratic forms of signature $(n,2)$.
We do not state their beautiful results as we do not require them for this paper.
We only note in their setting $E$ got replaced by ($\alpha\in\R$)
$$
E_2(\alpha; w_1, w_2):=\int_{\R^2}\sgn(t_1)\sgn(t_2+\alpha t_2)e^{-\pi(w_1-t_1)^2-\pi(w_2-t_2)^2} dt_1 dt_2.
$$
We note for comparison that their notation slightly differs from ours. We have, as $\lambda\to\infty$,
\begin{equation}\label{E2as}
E_2(\alpha; \lambda w_1,\lambda w_2) \sim \sgn(w_1)\sgn(w_2+\alpha w_2).
\end{equation}
Moreover
\begin{equation}\label{diffE}
	\Big(\partial_{w_1}^2+\partial_{w_2}^2 +2\pi\left(w_1\partial_{w_1}+w_2\partial_{w_2}\right)\Big)E_2\left(\alpha; w_1,w_2\right)=0,
\end{equation}
\begin{align}\label{E2diff1}
\partial_{w_2}E_2\left(\alpha; w_1,w_2\right) &= \frac{2}{\sqrt{1+\alpha^2}}e^{-\frac{\pi\left(w_2+\alpha w_1\right)^2}{1+\alpha^2}}E\left(\frac{\alpha w_2- w_1}{\sqrt{1+\alpha^2}}\right), \ \mathrm{and}
\\
\label{E2diff2}
\partial_{w_1} E_2\left(\alpha; w_1,w_2\right) &= 2e^{-\pi w_1 ^2}E\left(w_2\right) +\frac{2\alpha}{\sqrt{1+\alpha^2}}e^{-\frac{\pi\left(w_2+\alpha w_1\right)^2}{1+\alpha^2}}E\left(\frac{\alpha w_2- w_1}{\sqrt{1+\alpha^2}}\right).
\end{align}

\section{Generalized error integrals}
In this section, we introduce higher-dimensional analogies of the error function, following ideas of \cite{ABMP, Zwmult}.
\subsection{Definitions and basic properties}
The authors of \cite{ABMP} proposed a $3$-dimensional analogue of $E_2$, namely
\begin{multline*}
E_3^*(\alpha_1, \alpha_2, \alpha_3; w_1, w_2, w_3):=\int_{\R}
\sgn(t_1+\alpha_1 \alpha_2 t_2+\alpha_2 t_3)\sgn(t_2+\alpha_2\alpha_3t_3+\alpha_3 t_1)\\
\times\sgn(t_3+\alpha_3 \alpha_1 t_1+\alpha_1 t_2)e^{-\pi\left((w_1-t_1)^2+(w_2-t_2)^2+(w_3-t_3)^2\right)} dt_1 dt_2dt_3.
\end{multline*}
Here we define a modified version which is convenient for our explicit functions. After finishing this article, Pioline pointed out to the authors that there is an explicit map from the suggested higher dimensional $E_3$ function of \cite{ABMP}; however, our function may also offer some advantages as it seems easier to directly work with in at least some examples.
We define a generalized error function $E_N\colon R^{\frac{N(N-1)}{2}}\times \R^N \to \R$ by
\begin{align}\label{EN}
	&E_N(\alpha; w)\\
	&:=\int_{\R^N}\sgn(t_1)\sgn(t_2+\alpha_1t_1)\sgn(t_3+\alpha_2 t_1+\alpha_3 t_2)\cdots \sgn\left(t_N+\ldots+\alpha_{\frac{N(N-1)}{2}}t_{N-1}\right) e^{-\pi \|t-w\|_2^2} dt,\notag
\end{align}
where $\lVert a \rVert_2:=\sqrt{ a^Ta}$ denotes the Euclidian norm.
 Note that we use $\|\cdot\|_2$ for different dimensions, the meaning being clear from context.  Higher-dimensional $E_N$ collapse to lower-dimensional ones if certain $\alpha_j$ are $0$. For example,
\begin{align}
E_3(\alpha,0,0;w)
&=E_2(\alpha; w_1,w_2)E(w_3), \label{factor1}\\
E_3(0,\alpha,0;w)
&=E(w_2)E_2(\alpha;w_1,w_3).\label{factor2}
\end{align}

From now on, we restrict to $N=3$, however most of our statements hold for general $N$.
The following lemma, which generalizes (\ref{Eas}) and (\ref{E2as}), describes the asymptotic behavior of $E_N$, which is crucial in the construction of the appropriate completions of Lau and Zhou's functions.

\begin{lemma}\label{lem:2.1}
	For any $\alpha=(\alpha_1,\alpha_2,\alpha_3), w=(w_1, w_2, w_3)\in \R^3$, we have, as $\lambda\to\infty$,
	\begin{align*}
	{E}_N(\alpha; \lambda w)\sim\sgn\left(w_1\right)
	\sgn\left(w_2+\alpha_1w_1\right)\sgn\left(w_3+\alpha_{2}w_1+\alpha_3 w_2\right).
	\end{align*}
\end{lemma}
\begin{remark}
Throughout the paper, we write $E_N(\alpha;w)\sim\ast$ to mean that $E_N(\alpha;\lambda w)\sim\ast$ as $\lambda\to\infty$.
\end{remark}
\begin{proof}[Proof of Lemma \ref{lem:2.1}]
Changing variables yields	
\[
	{E}_3\left(\alpha;w_1, w_2, w_3\right)=
	\int_{\R^3}e^{-\pi t^TMt}\sgn\left(t_1+w_1\right)
	\sgn\left(t_2+ v_2\right)
	\sgn\left(t_3+v_3 \right) dt,
	\]
	where $t:=(t_1,t_2,t_3)^T, v_2:=w_2+\alpha_1w_1, v_3:=w_3+\alpha_2w_1+\alpha_3w_2$, and
	\begin{equation*}
	M:=\begin{pmatrix}
	1 +(\alpha_1\alpha_3-\alpha_2)^2 +\alpha_1^2 & -\alpha_1 -\alpha_3(\alpha_1\alpha_3-\alpha_2) & \alpha_1\alpha_3-\alpha_2 \\
	-\alpha_1 -\alpha_3(\alpha_1\alpha_3-\alpha_2) &\alpha_3^2 +1 & -\alpha_3\\
	\alpha_1\alpha_3-\alpha_2 &-\alpha_3 & 1
	\end{pmatrix}.
	\end{equation*}
	Note that $\det(M)=1$ and that $M$ is positive-definite.
	
	Then consider the difference
	\begin{align*}
	& {E}_3(\alpha; \lambda w_1,\lambda w_2,\lambda w_3)-\sgn\left(w_1\right)
	\sgn\left(w_2+\alpha_1w_1\right)
	\sgn\left(w_3 +\alpha_3 w_2 +\alpha_2 w_1\right)\\
	=&
	\int_{\R^3}e^{-\pi t^TMt}\big(
	\sgn\left(t_1+\lambda v_1\right)
	\sgn\left(t_2+ \lambda v_2\right)
	\sgn\left(t_3+\lambda v_3 \right)
	-\sgn\left(v_1\right)
	\sgn\left(v_2\right)
	\sgn\left(v_3 \right)\big)
	dt.
	\end{align*}
It is easily checked that the integrand vanishes whenever $ \lvert t_j \rvert < \lambda  \lvert v_j\rvert $ for all $j$. Denote the complement of the cube given by these three inequalities for $t$ by $\mathcal B(\lambda)$. Outside of $\mathcal B(\lambda)$ we bound the sum by 2 and obtain the following as an upper bound of the absolute value:
	\begin{align*}
	2\int_{\mathcal B(\lambda) }e^{-\pi t^TMt} dt.
	\end{align*}
	Since $M$ is positive-definite, this converges to $0$ as $\lambda \rightarrow \infty$ whenever $v_j\neq 0$, proving the statement in this case. If (at least) one $v_j$ vanishes, the integral expression for $ {E}_3$ vanishes which may be seen by changing $t_j\mapsto -t_j$.
\end{proof}

Certain sign-factors that occur throughout our investigation turn out to not quite have the correct shape. For this, the following elementary lemma, whose proof we skip, is useful.
\begin{lemma}\label{lemma3.2}
	For $a,b,c\in\R\backslash\{0\}$ and $\lambda_j\geq 0\  (j=1,2,3),\ \varepsilon\in\{\pm 1\}$, we have (unless $(\lambda_1,\lambda_2)=(0,0)$)
	\begin{equation}\label{eq:2sgn}
	\sgn(a)\sgn(b)= -\varepsilon+ \sgn(\lambda_1a+\lambda_2 \varepsilon b)(\varepsilon\sgn(a)+\sgn(b))
	\end{equation}
and (unless $(\lambda_1,\lambda_2, \lambda_3)=(0,0,0)$)
	\begin{equation}\label{eq:3sgn}
	\begin{split} &\sgn(a)\sgn(b)\sgn(c)   +\sgn(a)+\sgn(b)+\sgn(c)\\=& \sgn(\lambda_1a+\lambda_2b+\lambda_3 c)(\sgn(a)\sgn(b)+ \sgn(a)\sgn(c) +\sgn(b)\sgn(c)+ 1).
	\end{split}
	\end{equation}
\end{lemma}
\begin{remark}
	Applying Lemma \ref{lemma3.2} to $E_2(\alpha; w)$ gives, for $\alpha\neq 0$,
	\begin{align*}
		E_2(\alpha;w_1,w_2)=E(w_1)E(w_2)-\sgn(\alpha)E_2\left(\alpha^{-1}; w_2,w_1\right)+\sgn(\alpha).
	\end{align*}
	In the ``cuspidal case'' considered below, we must allow certain values to be $0$. To do so, the following lemma turns out to be useful.
\end{remark}

\begin{lemma}\label{la:CuspLimit}
	For $\alpha_1, \alpha_2, \alpha_3, w_1, w_2, w_3, w_4\in\R$ with $\alpha_3, \alpha_1\alpha_3-\alpha_2\neq 0$, we have
	\begin{align*}
		&\mathcal{E}_3(\alpha; w_1, w_2,w_3):=\lim_{T\to \infty} E_3\left(\alpha_1, T\alpha_2, T\alpha_3; w_1, w_2, Tw_3\right)\\
		&=\sgn(\alpha_3)E(w_1)+\delta \sgn(\alpha_3) E\left(\frac{w_2+\alpha_1 w_1}{\sqrt{\alpha_1^2+1}}\right)-\delta E\left(\frac{w_3+\alpha_2 w_1+\alpha_3 w_2}{\sqrt{\alpha_2^2+\alpha_3^2}}\right)
		\\
		&\hspace{10mm}
		-\sgn\left(\alpha_3w_3\right)\Bigg(1
		+\delta E_2\left(\alpha_1; w_1, w_2\right)+\delta E_2\left(\alpha_2 \alpha_3^{-1}; w_1, \alpha_3^{-1}w_3+w_2\right)
		\\
		&\hspace{10mm}
		+\delta E_2\left(\frac{\alpha_1\alpha_2+\alpha_3}{-\alpha_2+\alpha_1\alpha_3};\frac{w_2+\alpha_1 w_1}{\sqrt{\alpha_1^2+1}}, \frac{\sqrt{\alpha_1^2+1}}{-\alpha_2+\alpha_1\alpha_3}\left(w_3+
		\frac{(\alpha_1\alpha_3-\alpha_2)(\alpha_1w_2-w_1)}{\alpha_1^2+1}\right)\right)\Bigg),
		\end{align*}
		where $\delta:=\sgn(\alpha_3(\alpha_2-\alpha_1 \alpha_3))$.
\end{lemma}
\begin{proof}
We directly compute that
	\begin{align*}
		&E_3\left(\alpha_1, T\alpha_2, T\alpha_3; w_1, w_2, Tw_3\right)\\
		&\overset{T\to\infty}{\longrightarrow}\int_{\R^3}e^{-\pi \|t\|_2^2}\sgn(t_1+w_1)\sgn(t_2+\alpha_1 t_1+w_2+\alpha_1 w_1)\sgn(\alpha_2 t_1+\alpha_3 t_2+w_3+\alpha_2 w_1 +\alpha_3 w_2)dt.
	\end{align*}
	The integral over $t_3$ may now be computed to be $1$.
	
	To determine the remaining two-dimensional integral, we use Lemma \ref{lemma3.2}, with $ \lambda_1=|\alpha_2-\alpha_1 \alpha_3|,\ \lambda_2=|\alpha_3|$, $\lambda_3=1$, $a=\delta(t_1-w_1)$, $b=t_2+\alpha_1t_1+v_2$, and $c=\sgn(-\alpha_3)\alpha_2 t_1 -|\alpha_3|t_2+\sgn(-\alpha_3)v_3$, where $v_2:=w_2+\alpha_1 w_1$, $v_3:=w_3+\alpha_2 w_1+\alpha_3 w_2$, gives that the product of the signs equals (outside the zero set given by $abc=0$)
	\begin{align*}
	&\sgn(\alpha_3)\delta\Big(\sgn\left(\delta\left(t_1+w_1\right)\right)+\sgn\left(t_2+\alpha_1 t_1+v_2\right)+\sgn\left(\sgn(-\alpha_3)\alpha_2 t_1 -|\alpha_3|t_2+\sgn(-\alpha_3)v_3\right)\Big)\\
	&-\sgn(\alpha_3w_3)\Big(1+\sgn\left(\delta\left(t_1+w_1\right)\right)\sgn\left(t_2+\alpha_1 t_1+v_2\right)\\
	&\hspace{50mm}
	+\sgn\left(\delta\left(t_1+w_1\right)\right) \sgn\big(\sgn(-\alpha_3)\alpha_2 t_1-|\alpha_3| t_2+\sgn(-\alpha_3)v_3\big)\\
	&\hspace{50mm}
	+\sgn\left(t_2+\alpha_1 t_1 +v_2\right)\sgn\big(\sgn(-\alpha_3)\alpha_2 t_1-|\alpha_3| t_2+\sgn(-\alpha_3)v_3\big)\Big).
	\end{align*}	
	We compute all like integrals separately.
	The terms $-\sgn(\alpha_3w_3)$ and $\sgn(\alpha_3)\sgn(t_1+w_1)$ directly give $-\sgn(w_3\alpha_3)$ and
$\sgn(\alpha_3)E(w_1)$, respectively.

To consider the contribution from $\sgn(\alpha_3)\delta\sgn(t_2+\alpha_1 t_1+v_2)$, we define the orthonormal matrix $M_1:=(\alpha_1^2+1)^{-\frac12}\begin{psmallmatrix}
	\alpha_1&1\\-1&\alpha_1
	\end{psmallmatrix}$. Since $M_1$ is orthogonal, $\|M_1 t\|_2=\|t\|_2$. Thus, the contribution from this term gives
	$$
	\delta \sgn(\alpha_3)\int_{\R^2}e^{-\pi\|t\|_2^2}\sgn\left(t_1+\left(\alpha_1^2+1\right)^{-\frac12}v_2\right)dt_1 dt_2=\delta \sgn(\alpha_3) E\left(\frac{v_2}{\sqrt{\alpha_1^2+1}}\right).
	$$
	The contribution from $-\delta\sgn(\alpha_2 t_1+\alpha_3 t_2+v_3)$ is treated in exactly the same way, giving
\begin{align*}
-\delta E\left(\frac{v_3}{\sqrt{\alpha_2^2+\alpha_3^2}}\right).
	\end{align*}
	Next, we consider the product of two sgn-factors. The terms $-\delta\sgn(\alpha_3w_3)\sgn(t_1+w_1)\sgn(t_2+\alpha_1 t_1+v_2)$ and $-\delta\sgn(\alpha_3w_3)\sgn(t_1+w_1)\sgn(\alpha_2\alpha_3^{-1}t_1+t_2+\alpha_3^{-1}v_3)$ yield the contributions \\
	$-\delta\sgn(\alpha_3w_3)E_2(\alpha_1;w_1,w_2)$ and $-\delta\sgn(\alpha_3w_3)E_2(\alpha_2\alpha_3^{-1};w_1,\alpha_3^{-1}v_3-\alpha_2\alpha_3^{-1}w_1)$, respectively.\\
	Finally
	\begin{align*}
		&\sgn\left(w_3\right)\int_{\R^2}e^{-\pi\|t\|_2^2}\sgn\left(t_2+\alpha_1 t_1+v_2\right)\sgn\left(\alpha_2 t_1+\alpha_3 t_2+v_3\right)dt\\
		&=\sgn\left(w_3\right)\int_{\R^2}e^{-\pi\|t\|_2^2}\sgn\left(\left(M_1t\right)_1\sqrt{\alpha_1^2+1}+v_2\right)\sgn\left(\alpha_2 t_1+\alpha_3 t_2+ v_3\right)dt\\
		&=\sgn\left(w_3\right)\int_{\R^2}e^{-\pi\|t\|_2^2}\sgn\left(t_1+\left(\alpha_1^2+1\right)^{-\frac12}v_2\right)\sgn\left(\alpha_2\left(M^{-1}t\right)_1+\alpha_3\left(M^{-1}t\right)_2+v_3\right)dt\\
		&=-\delta\sgn\left(\alpha_3 w_3\right)E_2\left(\frac{\alpha_1\alpha_2+\alpha_3}{-\alpha_2+\alpha_1\alpha_3};\frac{v_2}{\sqrt{\alpha_1^2+1}}, \frac{\sqrt{\alpha_1^2+1}}{-\alpha_2+\alpha_1\alpha_3}\left(v_3-\frac{\alpha_1\alpha_2+\alpha_3}{\alpha_1^2+1}v_2\right)\right).
	\end{align*}
\end{proof}

We next show that the function $E_3$ satisfies a special differential equation.

\begin{lemma}\label{3.4}
	We have
	\[
	\sum_{j=1}^3\left(\partial^2_{w_j}+2w_j \partial_{w_j}\right) E_3(\alpha;w)=0.
	\]
\end{lemma}

\begin{proof}
	We write
	\begin{align*}
	E_3(\alpha;w)
	&=\int_{\R^2}\sgn\left(t_1\right)e^{-\pi\left(t_1-w_1\right)^2} \\
	&\quad \times \int_{\R^2}\sgn\left(t_2\right) \sgn\left(t_3+\left(\alpha_2-\alpha_1\alpha_3\right)t_1+\alpha_3t_2\right) e^{-\pi\left(\left(t_2-w_2-\alpha_1t_1\right)^2+\left(t_3-w_3\right)^2\right)}dt\\
	&=\int_\R\sgn(t_1)E_2\left(\alpha_3;w_2+\alpha_1t_1,w_3+\left(\alpha_2-\alpha_1\alpha_3\right)t_1\right)e^{-\pi(t_1-w_1)^2} dt_1.
	\end{align*}
	Applying the operator on the left-hand-side gives
	\begin{align*}
	&\int_\R\sgn(t_1)\frac{\partial}{\partial t_1} \left((-2\pi t_1)e^{-\pi(t_1-w_1)^2}\right)   E_2\left(\alpha_3;w_2+\alpha_1t_1,w_3+\left(\alpha_2-\alpha_1\alpha_3\right)t_1\right) dt_1 \\
	&\qquad+\int_\R\sgn(t_1)(-2\pi t_1)e^{-\pi(t_1-w_1)^2}\\
	&\qquad\times\left(\alpha_1E_2^{(1,0)}\left(\alpha_3;w_2+\alpha_1t_1,w_3+\left(\alpha_2-\alpha_1\alpha_3\right)t_1\right)\right.\\
	&\left.\qquad+\left(\alpha_2-\alpha_1\alpha_3\right)E_2^{(0,1)}\left(\alpha_3;w_2+\alpha_1t_1,w_3+\left(\alpha_2-\alpha_1\alpha_3\right)t_1\right) \right) dt_1\\
	&=-2\pi \int_\R\sgn(t_1)\frac{\partial}{\partial t_1} \left(t_1 e^{-\pi(t_1-w_1)^2} E_2\left(\alpha_3;w_2+\alpha_1t_1,w_3+\left(\alpha_2-\alpha_1\alpha_3\right)t_1\right) \right) dt_1=0,
	\end{align*}
	by \eqref{diffE} and the chain rule. This gives the claim.
	\end{proof}

\begin{lemma}
We have
\begin{align*}
&\partial w_3E_3(\alpha; w)
=\frac{2}{\sqrt{1+\alpha_3^2}}e^{-\frac{\pi\left(1+\alpha_3^2\right)}{1+\alpha_2^2+\alpha_3^2}\left(w_3+\alpha_3w_2+\alpha_2w_1\right)^2}\\
&\quad \times E_2\left(\frac{\alpha_2\alpha_3-\alpha_1\left(\alpha_3^2+1\right)}{\sqrt{\left(1+\alpha_3^2\right)\left(1+\alpha_2^2+\alpha_3^2\right)}};\frac{\left(w_2+\alpha_3w_2\right)\alpha_2-\left(1+\alpha_3^2\right)w_1}{\sqrt{1+\alpha_2^2+\alpha_3^2}},\frac{\alpha_3 w_3-w_2}{\sqrt{1+\alpha_3^2}}\right),\\
&\partial w_2E_3(\alpha; w)
=\frac{2\alpha_1 }{\sqrt{1+\alpha_3^2}}e^{-\frac{\pi\left(1+\alpha_3^2\right)}{1+\alpha_2^2+\alpha_3^2}\left(w_3+\alpha_3w_2+\alpha_2w_1\right)^2}\\
&\quad \times E_2\left(\frac{\alpha_2\alpha_3-\alpha_1\left(\alpha_3^2+1\right)}{\sqrt{\left(1+\alpha_3^2\right)\left(1+\alpha_2^2+\alpha_3^2\right)}};\frac{\left(w_2+\alpha_3w_2\right)\alpha_2-\left(1+\alpha_3^2\right)w_1}{\sqrt{1+\alpha_2^2+\alpha_3^2}},\frac{\alpha_3 w_3-w_2}{\sqrt{1+\alpha_3^2}}\right)\\
&+\frac2{\sqrt{1+\alpha_1^2}}e^{-\pi\left(\alpha_1w_1-w_2\right)^2}E_2\left(\frac{\alpha_2-\alpha_1\alpha_3}{\sqrt{1+\alpha_1^2}}; \frac{w_1-\alpha_1w_2}{\sqrt{1+\alpha_1^2}}, w_3\right),\\
&\partial w_1E_3(\alpha; w)
=\frac{2\alpha_1\left(1+\alpha_2-\alpha_3\right)}{\sqrt{1+\alpha_3^2}}e^{-\frac{\pi\left(1+\alpha_3^2\right)}{1+\alpha_2^2+\alpha_3^2}\left(w_3+\alpha_3w_2+\alpha_2w_1\right)^2}\\
&\quad \times E_2\left(\frac{\alpha_2\alpha_3-\alpha_1\left(\alpha_3^2+1\right)}{\sqrt{\left(1+\alpha_3^2\right)\left(1+\alpha_2^2+\alpha_3^2\right)}};\frac{\left(w_2+\alpha_3w_2\right)\alpha_2-\left(1+\alpha_3^2\right)w_1}{\sqrt{1+\alpha_2^2+\alpha_3^2}},\frac{\alpha_3 w_3-w_2}{\sqrt{1+\alpha_3^2}}\right)\\
&+\frac{2\left(\alpha_2-\alpha_1\alpha_3\right)}{\sqrt{1+\alpha_1^2}}e^{-\pi\left(\alpha_1w_1-w_2\right)^2}E_2\left(\frac{\alpha_2-\alpha_1\alpha_3}{\sqrt{1+\alpha_1^2}}; \frac{w_1-\alpha_1w_2}{\sqrt{1+\alpha_1^2}}, w_3\right)+2E_2\left(\alpha_2; w_2, w_3\right)e^{-\pi w_1^2}.
\end{align*}
\end{lemma}
\begin{proof}
In the proof of Lemma \ref{3.4}, we see that
$$
E_3(\alpha; w)=\int_\R \sgn\left(t_1\right)e^{-\pi\left(t_1-w_1\right)^2}E_2\left(\alpha_3; w_2+\alpha_1 t_1, w_3+\left(\alpha_2-\alpha_1\alpha_3\right)t_1 \right)dt_1.
$$
We apply first \eqref{E2diff1} to get
\begin{multline}\label{delw3}
\partial_{w_3}E_3(\alpha; w)\\
=\int_\R \sgn\left(t_1\right)e^{-\pi\left(t_1-w_1\right)^2}\frac{2}{\sqrt{1+\alpha_3^2}} e^{-\frac{\pi}{1+\alpha_3^2}\left(w_3+\alpha_2 t_1+\alpha_3w_2 \right)^2} E\left(\frac{\alpha_3 \left(w_3+\left(\alpha_2-\alpha_1 \alpha_3\right) t_1\right)-\left(w_2+\alpha_1 t_1\right)}{\sqrt{1+\alpha_3^2}} \right) dt_1\\
=\frac{2}{\sqrt{1+\alpha_3^2}}e^{-\frac{\left(1+\alpha_3^2\right)\pi}{1+\alpha_2^2+\alpha_3^2}\left(w_3+\alpha_3w_2+\alpha_2w_1\right)^2}\int_{\R}\sgn\left(t_1\right)e^{-\pi\left(\sqrt{1+\alpha_2^2+\alpha_3^2}t_1+\frac{1}{\sqrt{1+\alpha_2^2+\alpha_3^2}}\left(\left(w_3+\alpha_3w_2\right)\alpha_2-\left(1+\alpha_3^2\right)w_1\right)^2\right)} \\
E\left(\frac{\left(\alpha_2\alpha_3-\alpha_1\alpha_3^2-\alpha_1\right)t_1+\alpha_3w_3-w_2}{\sqrt{1+\alpha_3^2}}\right)dt_1.
\end{multline}
Making the change of variables $t_1\mapsto\frac{t_1}{\sqrt{1+\alpha_2^2+\alpha_3^2}}$ then gives the claim.

We next apply $\partial_{w_2}$. By \eqref{E2diff2}, we get
\begin{multline*}
\partial_{w_2}E_3(\alpha;w)= \int_{\R}\sgn(t_1)e^{-\pi(t_1-w_1)^2}\big( \alpha_1 \partial_{w_3}\left(E_2\left(\alpha_3;w_2+\alpha_1t_1,w_3+\left(\alpha_2-\alpha_1\alpha_3\right)t_1\right)\right)\\
+2e^{-\pi\left(w_2+\alpha_1t_1\right)^2}E\left(w_3+\left(\alpha_2-\alpha_1\alpha_3\right)t_1\right)\big) dt_1.
\end{multline*}
The first summand is computed as above.
The second term gives the claimed contribution by simplifying and
making the change of variables $t_1\mapsto\frac{t_1}{\sqrt{1+\alpha_1^2}}$.
Finally we apply $\partial_{w_1}$ to give, using integration by parts,
\begin{equation*}
\begin{aligned}
	&\quad \partial_{w_1}E_3\left(\alpha;w\right)
	\\
	&
	=2E_2\left(\alpha_2;w_2,w_3\right)e^{-\pi w_1^2}
	+\int_{\R}\sgn\left(t_1\right)e^{-\pi\left(t_1-w_1\right)^2}
	\alpha_1\frac{\partial}{\partial w_2}
	\left(E_2\left(\alpha_2;w_2+\alpha_1t_1,w_3+\left(\alpha_2-\alpha_1\alpha_3\right)t_1\right)\right)
	\\
	&
	+\left(\alpha_2-\alpha_1\alpha_3\right)\frac{\partial}{\partial w_3}
	\left(E_2\left(\alpha_2;w_2+\alpha_1t_1,w_3+\left(\alpha_2-\alpha_1\alpha_3\right)t_1\right)\right)dt_1.
\end{aligned}
\end{equation*}
From the above the result follows.
\end{proof}

\subsection{The function $E_3$ as a building block}
The theta functions of interest in Gromov-Witten theory are indefinite theta functions in which the summation conditions may be written in terms of sgn-functions. The following proposition shows how to turn their sgn-factors into functions satisfying Vign\'eras differential equation.

\begin{proposition}\label{th:Completions}
	For $N\in\N_0$, let $A=P^{-T}\operatorname{diag}(I_N,-I_3)P^{-1}\in \operatorname{Mat}_{N+3}(\R)$ be a symmetric matrix of signature $(N,3)$, $P\in \operatorname{GL}_{N+3}(\R)$, and assume that $a,b,c\in\R^{N+3}$ generate a 3-dimensional space of signature $(N_+,N_-)$ with respect to the bilinear form $\langle \cdot, \cdot \rangle$ given by $A^{-1}$. Then there exist $d,e,f\in\R^{N+3}$ and $\alpha_1,\alpha_2,\alpha_3\in\R$ (determined explicitly in Lemma \ref{la:explicitVectors} below) such that the following are true.

	\begin{enumerate}[leftmargin=*]
		\item For $(N_+,N_-)=(0,3)$ the map $X\mapsto E_3(\alpha_1, \alpha_2, \alpha_3; d^TPX,e^TPX,f^TPX$ satisfies Vign\'eras' differential equation for $\operatorname{diag}(I_N,-I_3)$, and for all $n\in\R^{N+3}$, we have
		\begin{align*}
		E_3\left(\alpha_1, \alpha_2, \alpha_3; d^Tn,e^Tn,f^Tn\right)\sim \sgn\left(a^Tn\right)\sgn\left(b^Tn\right)\sgn\left(c^Tn\right).
		\end{align*}
		\item For $(N_+,N_-)=(0,2)$ the map  $X\mapsto \mathcal{E}_3\left(\alpha_1,\alpha_2, \alpha_3; d^TPX,e^TPX, f^TPX\right)$
		satisfies Vign\'eras' differential equation for $\operatorname{diag}(I_N,-I_3)$ and for all $n\in\R^{N+3}$, we have
		\begin{align*}
		\mathcal{E}_3\left(\alpha_1,\alpha_2, \alpha_3; d^Tn,e^Tn, f^Tn\right)
		\sim \sgn\left(a^Tn\right)\sgn\left(b^Tn\right)\sgn\left(c^Tn\right).
		\end{align*}
	\end{enumerate}
\end{proposition}

Before proving Proposition \ref{th:Completions}, we require an auxiliary lemma.
\begin{lemma}\label{la:explicitVectors}
	Assume that $a,b,c\in\R^{N+3}$ generate a 3-dimensional space of signature $(0,3)$ or $(0,2)$ with respect to a symmetric bilinear form $\langle\cdot, \cdot\rangle$ of signature $(N,3)$ on $\R^{N+3}$. Then there exist pairwise orthogonal vectors $d,e,f\in\R^{N+3}\backslash\{0\}$ and scalars $\lambda,\mu,\nu\in\R\backslash\{0\}$, $\alpha_1,\alpha_2,\alpha_3\in\R$ \\
	\begin{equation}\label{abc}
	d=\lambda a,\qquad
	e+\alpha_1 d=\mu b,\qquad
	f+\alpha_2 d+\alpha_3 e =\nu c
	\end{equation}
	such that squares of the norms of $d,e,f$ are $(-2,-2,-2)$ for signature $(0,3)$ and $(-2,-2,0)$ for signature $(0,2)$, respectively. Explicitly, they can be defined (after permuting $a,b,c$ such that $\operatorname{span} \{a,b\}$ has signature $(0,2)$) as
	\begin{align*}
	\lambda&:=\sqrt{\frac{-2}{\|a\|^2}},\qquad\mu 
	:=\sqrt{\frac{-2\|a\|^2}{\|a\|^2 \|b\|^2 - \langle a,b\rangle^2}},
	\\
	\rho&:=-\frac{\|c\|^2}{2} + \frac{\langle a,c\rangle^2}{2\|a\|^2} -
	\frac14 \mu^2\left(\langle b,c\rangle -\frac{\langle a,b \rangle \langle a,c\rangle }{\|a\|^2}\right)^2,\qquad
	\nu:=\begin{cases}
	\lvert \rho\rvert  ^{-\frac12} \quad &\text{if } \rho\neq 0,\\
	1 \quad &\text{if } \rho= 0,\\
	\end{cases}
	\\
	\alpha_1 &:=-\frac{\langle a,b\rangle}{2}\lambda\mu
	,\qquad\alpha_2 :=-\frac{\langle a,c\rangle}{2} \lambda\nu
	,\qquad
	\alpha_3:=	-\frac{1}{2}\mu\nu\left(\langle b,c\rangle +\frac{\langle a,b \rangle \langle a,c\rangle \lambda^2}{2}\right).
	\end{align*}
	Furthermore, $\rho\geq 0$ vanishes if and only if the signature is $(0,2)$.
\end{lemma}
\begin{proof}
	Since by assumption $a$ and $b$ generate a negative-definite two-dimensional space, we have $\|a\|^2<0$.
	The definition of $\lambda$ directly yields that $\|d\|^2=-2$. Because the space spanned by $a$ and $b$ is negative-definite, we in particular have $\|a\|^2\|b\|^2-\langle a,b\rangle^2>0$.
	Therefore, $\mu$ is a well-defined positive number.
	We then compute, using that $\|d\|^2=-2$ and the definition of $\mu$,
	\begin{align*}
	\langle d,e\rangle &= \mu\langle d,b\rangle+2\alpha_1=\lambda\mu\langle a,b\rangle-\langle a,b\rangle\lambda   \mu=0,\\
	\|e\|^2&= \langle \mu b-\alpha_1 d,e  \rangle = \mu\langle b,\mu b-\alpha_1\lambda a\rangle=\mu^2\|b\|^2-\alpha_1\lambda\mu\langle a,b\rangle
	=\mu^2\left(\|b\|^2+\frac{1}{2}\langle a,b\rangle^2  \lambda^2   \right) =-2.
	\end{align*}
	Then by the choices of $\alpha_2$ and $\alpha_3$, we obtain that $f$, as defined in \eqref{abc}, is orthogonal to $d$ and $e$, using $\langle d,e\rangle =0$:
	\begin{align*}
	\langle d,f\rangle
	&= \langle d,\nu c - \alpha_2 d-\alpha_3 e\rangle
	= \nu\langle d, c\rangle +2\alpha_2 = \lambda\nu \langle a,c\rangle- \lambda\nu \langle a,c\rangle=0,\\
	\langle e,f\rangle
	&=\langle e ,\nu c - \alpha_2 d-\alpha_3 e\rangle
	=\nu\langle  e, c\rangle +2\alpha_3
	=\nu\langle  \mu b - \alpha_1 \lambda a, c\rangle - \mu\nu\left(\langle b,c\rangle +\frac{\langle a,b \rangle \langle a,c\rangle \lambda^2}{2}\right)\\
	&=\frac{\langle a,b\rangle}{2}\lambda^2\mu \nu\langle   a, c\rangle - \mu\nu \frac{\langle a,b \rangle \langle a,c\rangle \lambda^2}{2}=0.
	\end{align*}
	The above then yields
	\[
	\|c\|^2=\nu^{-2}\left(\|f\|^2-2\alpha_2^2-2\alpha_3^2\right).
	\]
	Finally we rewrite
	\begin{align*}
	\rho&=-\frac{	\|c\|^2}{2} + \frac{\langle a,c\rangle^2}{2\|a\|^2}
	+\frac14
	\frac{2\|a\|^2}{\|a\|^2\|b\|^2 - \langle a,b\rangle^2}
	\left(\langle b,c\rangle-\frac{ \langle a,b \rangle \langle a,c\rangle }{\|a\|^2}\right)^2
	\\
	&=-\frac{\|c\|^2}{2} + \frac{\langle a,c\rangle^2}{2\|a\|^2} -
	\frac14
	\left(-\frac{\|a\|^2\|b\|^2 - \langle a,b\rangle^2}{2\|a\|^2}\right)^{-1}
	\left(\langle b,c\rangle -\frac{\langle a,b \rangle \langle a,c\rangle }{\|a\|^2}\right)^2\\
	&=
	-\frac{\|c\|^2}{2} - \frac{\langle a,c\rangle^2}{4} \lambda^2-
	\frac14 \mu^2\left(\langle b,c\rangle +\frac{\langle a,b \rangle \langle a,c\rangle \lambda^2}{2}\right)^2.
	\end{align*}
	Using the definition of $\nu$ then yields
	\begin{equation*}
	\|f\|^2
	=\nu^2 \|c\|^2
	+2 \frac{\langle a,c\rangle^2}{4} \lambda^2\nu^2
	+ \frac{1}{2}\mu^2\nu^2\left(\langle b,c\rangle +\frac{\langle a,b \rangle \langle a,c\rangle \lambda^2}{2}\right)^2
	=-2\nu^2\rho
	=\begin{cases}
	0\quad &\text{if } \rho=0,\\
	-2\sgn(\rho)\quad &\text{if } \rho\neq 0.
	\end{cases}
	\end{equation*}
	This shows that $\rho$ vanishes if and only if $\Span \{d,e,f\}=\Span \{a,b,c\}$ has signature $(0,2)$. Since $\Span \{a,b,c\}$ is negative semi-definite, $\|f\|^2\leq 0$ and thus $\|f\|^2\in\{-2,0\}$. Therefore $\nu=\rho^{-\frac12}\geq 0$ if $\rho\neq 0$ such that $\nu$ agrees with the definition in the proposition.
\end{proof}

We are now ready to prove Proposition \ref{th:Completions}.
\begin{proof}[Proof of Proposition \ref{th:Completions}]
(1) We let $d,e,f\in\R^{N+3}$ and $\alpha_1,\alpha_2,\alpha_3\in\R$ be as in Lemma \ref{la:explicitVectors}. The definitions of $\alpha_1,\alpha_2,\alpha_3, d,e,f$, together with Lemma \ref{lem:2.1}, ensure that the asymptotics hold. Lemma \ref{la:explicitVectors} also implies that $P^Td,P^Te,P^Tf$  are pairwise orthogonal with squared norm $-2$ each (with respect to $D^{-1}=D:=\operatorname{diag}(I_N,-I_3)$). Combining this with Lemma \ref{3.4} and the chain rule then gives the claimed satisfaction of Vign\'eras' differential equation.		
		
		\noindent(2) Lemma \ref{la:explicitVectors} shows that $v,w\in\{d,e,f\}$ satisfy
		\begin{align*}
		-2\delta_{v,w}\left(1-\delta_{f,v}\right)=\langle v,w\rangle = v^T A^{-1}w =v^TPD^{-1}P^Tw=\left(P^Tv\right)^TDP^Tw.
		\end{align*}
		Therefore $(P^Td,P^Te,P^Tf)$ forms an orthogonal basis with norms squared $(-2,-2,0)$ with respect to $D^{-1}=D$. Note that there exists a subspace of signature $(1,3)$ with orthogonal basis $(d,e,f_+,f_-)$ such that $f=f_++f_-$. Setting (note that $ \|f_-\|=\|f_+\|$)
		
		\begin{align*}
		w_\varepsilon
		:=\frac{1}{\sqrt{2\|f_+\|^2\varepsilon}}\left( (1-\varepsilon)  f_++ (1+\varepsilon)  f_-\right),
		\end{align*}
		we compute
		$
		||w_\varepsilon||^2=-2.
		$
		 Therefore $(P^Td,P^Te,P^Tw_\varepsilon)$ is an orthogonal basis with norms squared $(-2,-2,-2)$ with respect to $D$. Just like the previous case, applying Lemma \ref{3.4} and the chain rule shows that
		
		\begin{equation*}
		X\mapsto E_3\left(\alpha_1, \frac{1}{\sqrt{2\|f_+\|^2\varepsilon}}\alpha_2, \frac{1}{\sqrt{2\|f_+\|^2\varepsilon}}\alpha_3; d^TPX,e^TPX,w_\varepsilon^TPX\right)
		\end{equation*}
		is a solution of Vigneras' differential equation. The scalar factors for $\alpha_2$ and $\alpha_3$ ensure that the function has the right asymtotic behaviour.

\end{proof}

\section{Lau and Zhou's explicit Gromov-Witten potential and simplifications for the proof of Theorem \ref{mainthm}}
In this section, we explicitly recall the functions arising in Gromov-Witten theory, which were studied by Lau and Zhou in \cite{LauZhou}, as well the explicit summation formulas for them by Cho, Hong, Kim, and Lau \cite{ChoHongKimLau}, and we start the investigation of their modularity properties. We assume throughout that $\mathbf{a}=(2,3,6)$ and study the function $W_q(2,3,6)$ defined in \cite{LauZhou}.  Namely, noting that in the notation of \cite{LauZhou} we have $q=q_d^{48}$, and writing the resulting coefficients as functions of $\tau$, by (3.29) of \cite{LauZhou}, $W_q(2,3,6)$ can be expanded as
\begin{equation}\label{W236Formula}
W_q(2,3,6)=q^{\frac18}X^2-q^{\frac1{48}}XYZ+c_Y(\tau)Y^3+c_Z(\tau)Z^6+c_{YZ2}(\tau)Y^2Z^2+c_{YZ4}(\tau)YZ^4,
\end{equation}
where
\begin{align}
c_Y(\tau)
&:=
q^{\frac{3}{16}}\sum_{n\geq0}(-1)^{n+1}(2n+1)q^{\frac{n(n+1)}2}
, \label{cy}\\
c_{YZ2}(\tau)
&:=
q^{-\frac1{12}}\sum_{n\geq a\geq0}\left((-1)^{n+a}(6n-2a+8)q^{\frac{(n+2)(n+1)}{2}-\frac{a(a+1)}{2}}+(2n+4)q^{n+an+1-a^2}\right)
,\label{cy2}\\
c_{YZ4}(\tau)
&:=
q^{-\frac{17}{48}}
\sum_{a,b\geq0 \atop{n\geq a+b}}
(-1)^{n+a+b}(6n-2a-2b+7)q^{\frac{(n+1)(n+2)}2-\frac{a(a+1)}2-\frac{b(b+1)}2}
,\label{cy4}\\
c_Z(q)
&:=q^{-\frac{5}{8}}\sum_{(n,a,b,c)\in T_1\cup T_2\cup T_3\cup T_6}
	(-1)^{n+a+b+c}\frac{(6n-2a-2b-2c+6)}{\eta(n,a,b,c)}q^{A(n,a,b,c)},\label{cz}
	\end{align}
	where
	\begin{align*}
	&A(n,a,b,c):=\binom{n+2}{2}-\binom{a+1}{2}-\binom{b+1}{2}-\binom{c+1}{2},\\
	&T_6:=\left\{(3a,a,a,a):a\in\N_0 \right\},\\
	&T_3:=\left\{(3a+k,a,a,a): a\in\N_0,k\in \N \right\},\\
	&T_2:=\left\{(a+b+c,a,b,c): a,b,c\in\N_0 \text{ such that } a<\min(b,c) \text{ or } a=c<b \right\},\\
	&T_1:=\left\{(a+b+c+k,a,b,c):k \in\N, a,b,c\in\N_0\text{ such that } a<\min(b,c) \text{ or } a=c<b \right\},\\
	&\eta(n,a,b,c):=j \text{ if }(n,a,b,c)\in T_j.
	\end{align*}
Note that the authors in \cite{ChoHongKimLau} and \cite{LauZhou} have an extra condition ``distinct'' in $T_1$. This turns out to just be a typo.
	
	In \cite{BRZ} modularity properties of $c_Y, c_{YZ2}$, and $c_{YZ4}$ were laid out and proven. We are thus left to investigate the hardest piece $c_Z$.
	The following lemma decomposes $c_Z$ into 3 simpler pieces.
\begin{lemma}
	We have
	\begin{align*}
	q^{\frac{5}{4}}c_Z(\tau)=F_1(\tau)-F_2(\tau)-\frac23 F_3(\tau),\qquad \text{where}
	\end{align*}
	\begin{align}
	F_1(\tau)&:=\left(\sum_{a,b,c\ge0\atop{k>0\atop{a<\min(b,c)}}}
	-\sum_{a,b,c<0 \atop{k\leq 0\atop{a\geq \max(b,c)}}} \right)
	(-1)^k (3k+2a+2b+2c+3)q^{{\frac{k^2}{2}+\frac{3k}{2}+ab+ac+ak+bc+bk+ck+a+b+c+1}},\label{F1}\\
	F_2(\tau)&:=\frac12
	\left(
	\sum_{a,b\geq0}-\sum_{a,b<0}
	\right)(6a+2b+3)q^{3a^2 + 2ab + 3a + b + 2},\label{F2}
	\\
	F_3(\tau)&:=\frac34 \left( \sum_{a,k\geq 0}-\sum_{a,k<0}\right) (-1)^k\left(2a+ k+ 1\right)q^{3a^2+3ak+3a+\frac{k^2}2 +\frac{3k}2+1}.\label{F3}
	\end{align}
\end{lemma}
\begin{proof}
	Let
	$$
	f_j(\tau):= \frac{1}{j}\sum_{(n,a,b,c)\in T_j} \widetilde{g}(n,a,b,c;\tau),\qquad
	\widetilde{g}(n,a,b,c;\tau)=(-1)^{n+a+b+c}(6n-2a-2b-2c+6)q^{A(n,a,b,c)}
	$$
	and $g(k,a,b,c;\tau):=\widetilde{g}(a+b+c+k,a,b,c;\tau)$
	and split
	\begin{align*}
	f_1(\tau)=f_{11}(\tau)+f_{12}(\tau),\qquad f_2(\tau)=f_{21}(\tau)+f_{22}(\tau)
	\end{align*}
	with
	\begin{align*}
	f_{11}(\tau)&:=\sum_{a,b,c\ge0, k>0\atop{a<\min(b,c)}} g(k,a,b,c;\tau),\qquad\quad
	f_{12}(\tau):=\sum_{a,b,c\geq0, k>0\atop{ c=a<b} } g(k,a,b,c;\tau),\\
  f_{21}(\tau)&:=\frac12\sum_{a,b,c\ge0, k=0\atop{a<\min(b,c)}} g(k,a,b,c;\tau), \qquad
	f_{22}(\tau):=\frac12 \sum_{a,b,c\geq0, k=0\atop{c=a<b} } g(k,a,b,c;\tau).
	\end{align*}
	Note that
	\begin{align*}
		\widetilde{g}\left(-n-3,-a-1,-b-1,-c-1\right)
		&=\widetilde{g}\left(n,a,b,c\right),\\
		g\left(-k,-a-1,-b-1,-c-1\right)
		&=g(k,a,b,c),
	\end{align*}
  which we use repeatedly.  We now compute
	\begin{align*}
	=\frac12 \left(\sum_{a,b,c\ge0, k>0\atop{a<\min(b,c)}}
	-\sum_{a,b,c,k< 0\atop{a>\max(b,c)}} \right)
	g(k,a,b,c;\tau)
	=\frac12 \left(\sum_{a,b,c\ge0, k>0\atop{a<\min(b,c)}}
	-\sum_{a,b,c<0, k\leq 0\atop{a>\max(b,c)}} \right)
	g(k,a,b,c;\tau) -f_{21}(\tau),
	\end{align*}
	and similarly
	\begin{align*}
	f_{12}(\tau)+2f_{22}(\tau)&=
	\sum_{a,b,c\geq0, k>0\atop{ c=a<b}} g(k,a,b,c;\tau)
	+\sum_{a,b,c\geq0, k=0\atop{ c=a<b}} g(k,a,b,c;\tau)
	=\sum_{a,b,c\geq0, k\geq 0\atop{ c=a<b}} g(k,a,b,c;\tau)
	\\&
	=\frac12 \sum_{a,b,c,k\geq 0\atop{ a=\min(b,c)}}g(k,a,b,c;\tau)
	-\frac12\sum_{a,b,c,k\geq 0\atop{ c=a=b}} g(k,a,b,c;\tau)
	\\&
	=-\frac12 \sum_{a,b,c<0, k\leq 0\atop{ a=\max(b,c)}}g(k,a,b,c;\tau)
	-\frac{3}{2}f_3(\tau)-\frac{6}{2}f_6(\tau).
	\end{align*}

	Therefore
	\begin{align*}
	&\quad f_1(\tau)+f_2(\tau)+f_{22}(\tau)+\frac{3}{2}f_3(\tau)+3f_6(\tau)
	= f_{11}(\tau)+f_{12}(\tau)+f_{21}(\tau)+2f_{22}(\tau)+\frac{3}{2}f_3(\tau)+3f_6(\tau)
	\\&=\frac12 \left(\sum_{a,b,c\ge0, k>0\atop{a<\min(b,c)}}
	-\sum_{a,b,c<0, k\leq 0\atop{a>\max(b,c)}} \right)
	g(k,a,b,c;\tau) -
	\frac12 \sum_{a,b,c<0, k\leq 0\atop{ a=\max(b,c)}}g(k,a,b,c;\tau)
	\\&=
	\frac12 \left(\sum_{a,b,c\ge0, k>0\atop{a<\min(b,c)}}
	-\sum_{a,b,c<0, k\leq 0\atop{a\geq \max(b,c)}} \right)
	g(k,a,b,c;\tau)=F_1(\tau).
	\end{align*}
	
	For $f_{22}$, we find that
	\begin{align*}
	f_{22}(\tau)&=\frac12 \sum_{a,b,c\geq 0\atop{ c=a<b}}g(0,a,b,c;\tau)
	=\frac14 \left(\sum_{a,b,c\geq 0\atop{ c=a<b}}-
	\sum_{a,b,c< 0\atop{ c=a>b}}
	\right)g(0,a,b,c;\tau)
	\\&=\frac14 \left(\sum_{a,b,c\geq 0\atop{ c=a\leq b}}-
	\sum_{a,b,c< 0\atop{ c=a>b}}
	\right)g(0,a,b,c;\tau)
	-\frac32f_6(\tau).
	\end{align*}
	Making the change of variables $b\mapsto b+a$, we obtain that the first sum equals
	\begin{equation*}
	\frac14\left(\sum_{a,b\geq 0}-\sum_{a,b<0}\right) g(0,a,a+b,a;\tau)
	=
	\frac14
	\left(
	\sum_{a,b\geq0}-\sum_{a,b<0}
	\right)(12a+4b+6)q^{3a^2 + 2ab + 3a + b + 2}=F_2(\tau).
	\end{equation*}
	Finally, we compute
	\begin{align*}
	3f_3(\tau)+3f_6(\tau)&=\sum_{a=b=c\geq 0 \atop{k >0 }} g(k,a,b,c)
	+\frac12 \sum_{a=b=c\geq 0 \atop{k =0 }}g(k,a,b,c;\tau)
	\\&=\frac12 \left(\sum_{a=b=c\geq 0 \atop{k >0 }}-\sum_{a=b=c< 0 \atop{k <0 }}\right) g(k,a,b,c;\tau)
	+\frac12 \sum_{a=b=c> 0 \atop{k =0 }}g(k,a,b,c;\tau)
	\\&=\frac12 \left(\sum_{a=b=c\geq 0 \atop{k \geq 0 }}-\sum_{a=b=c< 0 \atop{k < 0 }}\right) g(k,a,b,c;\tau)
	\\&=\frac12 \left( \sum_{a,k\geq 0}-\sum_{a,k<0}\right) (-1)^k(12a+6k+6)q^{3a^2+3ak+3a+\frac{k^2}{2} +\frac{3k}{2}+1}=4F_3(\tau).
	\end{align*}
	Combining completes the proof.
\end{proof}

\begin{lemma}\label{LemmaF2F3ModCompletions}
The functions $F_2$ and $F_3$ have modular completions.
\end{lemma}
\begin{proof}
We view $F_2$ as derivatives of indefinite theta series with additional Jacobi variables (where $\zeta_j:=e^{2\pi iz_j}$)
\begin{equation*}
F_2(z_1,z_2;\tau):=\zeta_2^{\frac32}\left(\sum_{a,b\ge0}-\sum_{a,b<0} \right) (-1)^a\zeta_1^a\zeta_2^b q^{\frac{3a(a+1)}{2}+ab}.
\end{equation*}
Then
\begin{equation*}
F_2(\tau)=\frac{q^{\frac12}}{4}\left[\left(12\zeta_1\frac{\partial}{\partial \zeta_1}+4\zeta_2\frac{\partial}{\partial \zeta_2}+3 \right)F_2(z_1,z_2;\tau) \right]_{\zeta_1=-1\atop{\zeta_2=q\atop{q\mapsto q^2}}}.
\end{equation*}
Define
\[
\widehat F_2(z_1,z_2;\tau):=F_2(z_1,z_2;\tau)+\frac{i}{2}\sum_{k=0}^{2}\zeta_2^k\vartheta(z_1+k\tau;3\tau)R(3z_2-z_1-k\tau;3\tau),
\]
where (with $z=x+iy$)
\begin{align*}
\vartheta(z; \tau)&:=\sum_{n\in\frac12+\Z} e^{2\pi in\left(z+\frac12\right) }q^{\frac{n^2}{2}} =
-i q^{\frac{1}{8}} e^{-\pi i z} \prod_{n \geq 1} \left(1 - q^n\right)
\left(1 - e^{2 \pi i z} q^{n-1}\right) \left(1 - e^{-2 \pi i z} q^n\right),\\
R(z; \tau)&:=\sum_{n\in\frac12+\Z}\left(\sgn(n)-E\left(\left(n+\frac{y}{v}\right)\sqrt{2v}\right)\right)
(-1)^{n-\frac12} q^{-\frac{n^2}{2}} e^{-2\pi inz}.
\end{align*}
Setting
\[
\widehat{F}_2(\tau):=\frac{q^{\frac12}}{4}\left[\left(12\zeta, \frac{\partial}{\partial\zeta_1}+4\zeta_2\frac{\partial}{\partial\zeta_2}+3\right)\widehat{F}\left(z_1, z_2; \tau\right)\right]_{{\zeta_1=1\atop{\zeta_2=q}}\atop{q\mapsto q^2}}
\]
we have $\widehat{F}_2^+=F_2$.
Using \cite{Zwmult} we see that we have, for $\begin{psmallmatrix}a&b\\c&d \end{psmallmatrix}\in\SL_2(\Z)$ and $n_1,n_2,m_1,m_2\in\Z$,
\begin{align*}
\widehat F_2\left(\frac{z_1}{c\tau+d},\frac{z_2}{c\tau+d};\frac{a\tau+b}{c\tau+d}\right)
&=(c\tau+d)e^{\frac{\pi ic\left(-3z_2^2+2z_1z_2\right)}{c\tau+d}}\widehat{F}_2\left(z_1,z_2;\tau\right),\\
\widehat{F}_2\left(z_1+n_1\tau+m_1, z_2+n_2\tau +m_2\right)
&=(-1)^{n_1+m_1}\zeta_2^{3n_2-n_1} \zeta_1^{-n_2} q^{\frac{3n_2^2}{2}-n_1n_2} \widehat{F}_2\left(z_1,z_2;\tau\right).
\end{align*}
From this one can then derive the transformation law of $F_2(z_1+\frac12,z_2+\frac{\tau}{2};\tau)$.
Taking the appropriate derivatives with respect to $z_1$ and $z_2$ then gives additional terms involving $1,\tau,\tau^2$ which can be removed with the help of $1/v$-terms (or using powers of the weight 2 Eisenstein series). This yields the modular completion.
The shadow of $F_2(\tau)$ can be determined using \eqref{E2diff1} and \eqref{E2diff2} for $\widehat{F}_3\left(z_1, z_2; \tau\right)$ and then applying the appropriate Jacobi derivatives.

We next turn to $F_3$ and define the Jacobi version of $F_3$
\begin{equation*}
F_3(z_1,z_2;\tau):=\left(\sum_{a,b\ge0}-\sum_{a,b<0} \right) \zeta_1^a\zeta_2^b q^{3a^2+\frac{b^2}{2}+3ab}.
\end{equation*}
Then
\begin{equation*}
F_3(\tau)=\frac{3q}{4}\left[\left(2\zeta_1\frac{\partial}{\partial \zeta_1}+\zeta_2\frac{\partial}{\partial \zeta_2}+1 \right)F_3(z_1,z_2;\tau) \right]_{\zeta_1=q^3\atop{\zeta_2=-q^{\frac32}}}.
\end{equation*}
Writing
\[
F_3(z_1,z_2;\tau)=\frac12 \sum_{a,b\in\Z} \left(\sgn\left(a+\frac12\right)+\sgn\left(b+\frac12\right)  \right)\zeta_1^a\zeta_2^bq^{\frac12Q(a,b)}
\]
with $Q(a,b):=6a^2+b^2+6ab$, we obtain the completion
\[
\widehat F_3(z_1,z_2;\tau):=\frac12\sum_{a,b\in\Z}\left(E\left(\sqrt{v}\left(\sqrt{6}a+\frac12+y_1\right)\right) -E\left(\sqrt{v}\left(b+\frac12+y_2\right)\right) \right)\zeta_1^a\zeta_2^bq^{\frac12Q(a,b)}.
\]
The proof then follows as before.
\end{proof}

\section{An indefinite theta function of signature $(1,3)$}

A key in understanding $F_1$ is the following indefinite theta function
$$
\Theta(z;\tau):=\Theta_{0,\Z^4, A,  \widehat{p}, 0} (z;\tau),
$$
where $A:=
\left(\begin{smallmatrix}
1&1&1&1\\
1&0&1&1\\
1&1&0&1\\
1&1&1&0
\end{smallmatrix}\right)$, and
\begin{align*}
\widehat{p}(\ell)&:=-E\left(\ell_1\right)
-E\left(\sqrt{2}\ell_2\right)
+E\left(\ell_3-\ell_1\right) E_2\left(\frac{1}{\sqrt{3}}; \ell_1, \frac{1}{\sqrt{3}}\left(-\ell_2-\ell_3+2\ell_4\right) \right)
\\&\quad+E_3\left(\frac{1}{\sqrt{3}},-\frac{1}{\sqrt{2}},-\sqrt{\frac32};\ell_3-\ell_2,\frac{1}{\sqrt{3}}(-\ell_2-\ell_3+2\ell_4),\sqrt{\frac23}(\ell_2+\ell_3+\ell_4) \right)
\\&\quad  +\sgn\left(\ell_1+\ell_2+\ell_3\right)\Big(E\left(\ell_1\right)E\left(\ell_3-\ell_2\right)
-E_2\left(1; \ell_1, -\ell_1-2\ell_2 \right)
+E_2\left(-1;  \ell_3-\ell_2, \ell_2+\ell_3\right)+1\Big)
\\&\quad+\sgn\left(\ell_1+\ell_2+\ell_4\right)\Big(E\left(\ell_1\right)E\left(\ell_4-\ell_2\right)
-E_2\left(1; \ell_1, -\ell_1-2\ell_2 \right)
+E_2\left(-1;  \ell_4-\ell_2, \ell_2+\ell_4\right)+1\Big).
\end{align*}

\begin{theorem}\label{generalmod}
The function $\Theta$ transforms like a vector-valued Jacobi form.
\end{theorem}
There are two main steps that have to be made: convergence and showing that $\widehat{p}$ satisfies Vign{\'e}ras' differential equation.

\begin{proposition}\label{prop:convergence}
	The theta series $\Theta_{0,\Z^4, A, p,0}(z;\tau)$, as well as its modular completion $\Theta_{0,\Z^4, A, \widehat{p},0}(z;\tau)$ converges absolutely and uniformly on compact subsets of $\{(z,\tau)\in \C^n\times \C:B(c_j, \frac{y}{v})\not\in\Z \ (j\in\{0,1,2,3\})\}$. Here
	\begin{align*}
	p(\ell)
	&:=\Big(\sgn\left(c_0^T\ell\right)+\sgn\left(c_{1}^T\ell\right)\Big)
	\Big(\sgn\left(c_0^T\ell\right)+\sgn\left(c_{2}^T\ell\right)\Big)
	\Big(\sgn\left(c_1^T\ell\right)+\sgn\left(c_{3}^T\ell\right)\Big)
	\\&=
	\left(\sgn\left(\ell_1\right)+\sgn\left(\ell_2\right) \right)\left(\sgn\left(\ell_2\right)+\sgn\left(\ell_3-\ell_2\right) \right) \left(\sgn\left(\ell_1\right)+\sgn\left(\ell_4-\ell_2\right) \right)
	\end{align*}
	with $c_0:=(0,1,0,0)^T, c_1:=(1,0,0,0)^T, c_2:=(0,-1,1,0)^T$ and $c_3:=(0,-1,0,1)^T$.
\end{proposition}
\begin{proof}
	We begin by proving (absolute local uniform)  convergence of the holomorphic theta series
	$\Theta_{0,\Z^4, A, p,0}$, so that it suffices to additionally prove the convergence of $\Theta_{0,\Z^4, A, \widehat{p}- p,0}=\Theta_{0,\Z^4, A, \widehat{p},0}- \Theta_{0,\Z^4, A, p,0}$. We also note that for $(z,\tau)$ lying in the stated range, we have $B(c_j,n+\frac{y}{v})\neq 0$ for all $n\in \Z$.
	
	The proof of the convergence of the holomorphic theta function is a straightforward generalization of the proof by Zwegers \cite{Zw} and Alexandrov, Banerjee, Manschot, and Pioline \cite{ABMP}. 
	We rewrite
	\begin{align*}
	\Big(\sgn\left(c_k^T\ell\right)+\sgn\left(c_{j}^T\ell\right)\Big)
	=\Big(\sgn\left((A^{-1}c_k)^TA\ell\right)+\sgn\left((A^{-1}c_{j})^TA\ell\right)\Big)
	\end{align*}
	in $p$ and observe that
	\begin{align*}
	\Big\lvert q^{Q(n)} \zeta^{An}\Big\rvert =\exp\Big(-\pi v n^TAn -2\pi y^TAn\Big)
	=\exp\left(-2\pi Q\left(\sqrt{v}\left(n+\frac{y}{v}\right)\right)\right) \exp \left(\frac{2\pi Q(y)}{v}\right),
	\end{align*}
	so that
	\begin{align*}
	\sum_{n\in \mu+L} \left\lvert {p}\left(\sqrt{v}\left(n+\frac{y}{v}\right)\right)
	q^{Q(n)} \zeta^{An}
	\right \rvert
	=
	e^{\frac{2\pi Q(y)}{v}}
	\sum_{n\in \mu+L} \left\lvert {p}\left(\sqrt{v}\left(n+\frac{y}{v}\right)\right)
	\right \rvert
	e^{-2\pi Q\left(\sqrt{v}\left(n+\frac{y}{v}\right)\right)}.
	\end{align*}
	
	Therefore we need to investigate
	\begin{align*}
	\sum_{m \in \Lambda} {p}(m)e^{-2\pi Q(m)}
	\end{align*}
	for $\Lambda$ some lattice. Note that $B(c_j,m)=c_j^TAm=(Ac_j)^Tm$. Since the $\sgn(c_j^TAm)$ do not vanish, $p(m)\neq 0$ only if all $\sgn(c_j^TAm)$ are equal. Therefore we obtain $B(c_j,m)B(c_k,m)> 0$ for all $j,k\in\{0,1,2,3\}$ whenever $p(m)\neq 0$.  The matrix
	\begin{align*}
	\begin{pmatrix}
	B(m,m) & B(c_0,m) & B(c_1,m)& B(c_2,m)& B(c_3,m) \\
	B(c_0,m) & B(c_0,c_0) & B(c_0,c_1)& B(c_0,c_2)& B(c_0,c_3) \\
	B(c_1,m) & B(c_0,c_1) & B(c_1,c_1)& B(c_1,c_2)& B(c_1,c_3) \\	
	B(c_2,m) & B(c_0,c_2) & B(c_1,c_2)& B(c_2,c_2)& B(c_2,c_3) \\	
	B(c_3,m) & B(c_0,c_3) & B(c_1,c_3)& B(c_2,c_3)& B(c_3,c_3)
	\end{pmatrix}
	\end{align*}
	has determinant zero. We refer to the bottom right $4\times 4$ block as $G$, which has negative determinant. Both of these statements come from the fact that $\operatorname{span}\{c_0,c_1,c_2,c_3\}$ has signature $(1,3)$. Furthermore, a Laplace expansion along the first column and then another along the first row lets us write the determinant as
	\begin{align*}
	0 \geq B(m,m)\det(G)&-\sum_{k=0}^{3} B(c_k,m)^2 \det((G_{p,q})_{p,q\neq k})
	\\ &-2\sum_{0\leq j<k\leq 3}(-1)^{k+j} B(c_j,m)B(c_k,m) \det\big(	(G_{p,q})_{p\neq k, q\neq j}\big).
	\end{align*}
	Now note that $ \det((G_{p,q})_{p,q\neq k})\leq 0$ since the space $\operatorname{span}\{c_p; p\neq k\}$ has signature $(0,3)$ or $(0,2)$ and the $\sgn\left(\det\left(	(G_{p,q})_{p\neq k, q\neq j}\right)\right)=(-1)^{k+j+1}$ by direct calculation. Therefore, whenever it is $B(c_j,m)B(c_k,m) \geq 0$ for all $j,k\in\{0,1,2,3\}$ we obtain
\begin{align*}
B(m,m)= \det(G)^{-1}&\Bigg(
\sum_{k=0}^{3} B(c_k,m)^2 \det((G_{p,q})_{p,q\neq k})
\\ &\quad-2\sum_{0\leq j<k\leq 3} B(c_j,m)B(c_k,m)  \left\lvert\det\big(	(G_{p,q})_{p\neq k, q\neq j}\big)\right\rvert
\Bigg)
\\= \lvert\det(G)\rvert^{-1}&\Bigg(
\sum_{k=0}^{3} B(c_k,m)^2 \lvert\det((G_{p,q})_{p,q\neq k})\rvert
\\ &\quad+2\sum_{0\leq j<k\leq 3} \left\lvert B(c_j,m)B(c_k,m)  \det\big(	(G_{p,q})_{p\neq k, q\neq j}\big)\right\rvert
\Bigg) =:K(m)\geq 0.
\end{align*}
If $p(m)\neq 0$, then at most two $B(c_j,m)$ vanish, such that some terms in the second sum do not vanish and the inequality is strict. Therefore, with $p(m)\leq 8$, we obtain
\begin{align*}
\sum_{m \in \Lambda} p(m)e^{-2\pi Q(m)} \leq 8\sum_{m \in \Lambda}e^{-2\pi K(m)}.
\end{align*}
Now since none of the determinants in the second sum of $K(m)$ vanishes and the $B(c_j,m)$ do not vanish, we obtain for some constant $c>0$ that
\begin{align*}
K(m)\geq c\min(\lvert B(c_j,m)\rvert),
\end{align*}
which yields exponential decay of $e^{-2\pi K(m)}$ as $\lVert m\rVert \rightarrow \infty$ and thus the convergence (uniform on compact sets with respect to translations of $\Lambda$) of 
\begin{align*}
\sum_{m \in \Lambda} p(m)e^{-2\pi Q(m)} \leq 8\sum_{m \in \Lambda}e^{-2\pi K(m)}< \infty.
\end{align*}

	Next we treat the difference between the holomorphic part and the modular completion. By multiplying out $p$ and coupling the terms of $p$ and $\widehat{p}$ appropriately, we get a sum of series over terms which have either the shape
	\begin{align*}
	\left(-E_3\left(\alpha; B(d_1,m),B(d_2,m),B(d_3,m)\right)
	+\prod_{j\in \{0,1,2,3\}\backslash\{\ell\}} \sgn(B(c_j,m))\right) q^{Q(n)} e^{2\pi i B(z,n)},
	\end{align*}
	where 
	$d_1$, $d_2+\alpha_1 d_1$, $d_3+\alpha_2 d_1 + \alpha_3 d_2$ are the $c_j$ appearing in the product or the same shape with $E_3$ replaced by $\mathcal{E}_3$ or
	\begin{align*}
	\left(\sgn(B(c_\ell,m))-E(B(c_\ell,m))\right)\sgn(B(c_k,m))^2q^{Q(n)} e^{2\pi i B(z,n)}
	\end{align*}
	for some $(k,\ell)\in\{(0,1),(0,3),(1,0),(1,2)\}$.
	In a similar fashion as in the proof of Theroem 4.2 of \cite{ABMP},
	one can decompose each of these terms into a sum of integrals decaying square-exponentially in some directions $\{c_{j_1},\dots, c_{j_k}\}$ (i.e., it grows like $e^{-\pi \sum_k B(c_{j_k},m)^2}$ in additionto the general factor $e^{-\pi 2Q(m)}$). By combining the integrals of the same decay from different terms, one obtains cancellation of the sign-terms whenever the integrals times do not decay. This gives convergence of the the theta function. Further details can also be found in the second author's doctoral thesis \cite{JonasPhD}.
	Therefore the theta series $\Theta_{0,\Z^4, A, p-\widehat{p},0}$ and $\Theta_{0,\Z^4, A,\widehat{p},0}$ converge.
\end{proof}

	We next turn to proving that Vign\'eras' differential equation is satisfied in our situtaion.
\begin{proposition}\label{prop:VigApprox}
	The function $\ell\mapsto \widehat{p}(P\ell)$ is a solution of Vign\'eras' differential equation with respect to $D:=\operatorname{diag}(1,-1,-1,-1)$. It approximates $p$
	if $\ell_1,\ell_2\neq 0$. 
\end{proposition}

\begin{proof}
	Our approach is to split $p(\ell)$ into various terms, which we treat separately using Proposition \ref{th:Completions} and Lemma \ref{lemma2.2}.
	Multiplying the product of signs out, we obtain, using that $\ell_1$ and $\ell_2$ do not vanish
	\begin{align*}
	&\sgn\left(\ell_1\right)+\sgn\left(\ell_2\right)+\sgn\left(\ell_3-\ell_2\right)+\sgn\left(\ell_4-\ell_2\right)+\sgn\left(\ell_2\right)\sgn\left(\ell_3-\ell_2\right)\sgn\left(\ell_4-\ell_2\right)
	\\
	&
	+\sgn\left(\ell_1\right)\sgn\left(\ell_2\right)\sgn\left(\ell_3-\ell_2\right)+\sgn\left(\ell_1\right)\sgn\left(\ell_2\right)\sgn\left(\ell_4-\ell_2\right)+\sgn\left(\ell_1\right)\sgn\left(\ell_3-\ell_2\right)\sgn\left(\ell_4-\ell_2\right)
	.
	\end{align*}
	We first compute $
A^{-1}=\left(
\begin{smallmatrix}
-2&1&1&1\\
1&-1&0&0\\
1&0&-1&0\\
1&0&0&-1
\end{smallmatrix}\right).$
	
The single sign factors are treated as in Zwegers' thesis \cite{Zw} (see the description in Section \ref{ex.indef.theta}), yielding the following function
\[
E\left(\ell_1\right)+E\left(\sqrt{2}\ell_2\right)+E\left(\ell_3-\ell_2\right)+E\left(\ell_4-\ell_2\right)
\sim \sgn\left(\ell_1\right)+\sgn\left(\ell_2\right)+\sgn\left(\ell_3-\ell_2\right)+\sgn\left(\ell_4-\ell_2\right),
\]
where each of the summands on the left satisfies Vign\'eras differential equation in $(\ell_1,\ell_2,\ell_3,\ell_4)$ with respect to $\operatorname{diag}(1,-I_3)$ as can be verified directly.
Thus we are left to consider $3$ sign factors.

We start with $\sgn(\ell_1)\sgn(\ell_3-\ell_2)\sgn(\ell_4-\ell_2)$ and set
  	\[
		v_1:=(1,0,0,0)^T,\qquad v_2:=(0,-1,1,0)^T,\qquad v_3:=(0,-1,0,1)^T.
		\]
		Then, with $\langle a,b\rangle =a^TA^{-1}b$, we obtain
		\[
		\|v_j\|^2=-2,\qquad \langle v_1,v_2\rangle=\langle v_1,v_3\rangle=0,\qquad \langle v_2,v_3\rangle=-1.
		\]
		This easily gives that the corresponding signature is $(0, 3)$.
		We plug these into (\ref{la:explicitVectors}) to obtain
		\[
		\lambda=1,\qquad
		\mu = 1,\qquad
		\nu=\frac{2}{\sqrt{3}},\qquad
		\alpha_1=0, \qquad
		\alpha_2=0,\qquad
		\alpha_3=\frac{1}{\sqrt{3}},
	\]
\[
		d=v_1,\qquad
		e= v_2,\qquad
		f= \frac{2}{\sqrt{3}}v_3- \frac{1}{\sqrt{3}} w_2.
		\]
		Lemma \ref{lem:2.1} then yields that
		\[
		E_3\left(0,\frac{1}{\sqrt{3}},0;\ell_1,\ell_3-\ell_2,\frac{1}{\sqrt{3}} (-\ell_2-\ell_3+2\ell_4) \right)\sim
		\sgn(\ell_1)\sgn(\ell_3-\ell_2)\sgn(\ell_4-\ell_2)
		\]
		and $X\mapsto E_3(0,\frac{1}{\sqrt{3}},0; v_1^TP,v_2^TXP,3^{-\frac12}(2v_3^T-v_2^T)P)$.
	      The claim then follows using (\ref{factor2}).
	
		We next turn to the case $\sgn(\ell_2)\sgn(\ell_3-\ell_2)\sgn(\ell_4-\ell_2)$ and set
		\[
		v_1:=(0,-1,1,0)^T,\quad v_2:=(0,-1,0,1)^T,\quad v_3:=(0,1,0,0)^T.
		\]
		Then
		\[
		\|v_1\|^2=\|v_2\|^2=-2,\quad\|v_3\|^2=\langle v_1,v_2\rangle=-1,\quad\langle v_1,v_3\rangle=\langle v_2,v_3\rangle=1.
		\]
		We plug these into Lemma \ref{la:explicitVectors} to obtain
		\[
		\lambda=1,\qquad
			\mu = \frac{2}{\sqrt{3}},\qquad
	  \nu =\sqrt{6}, \qquad
		\alpha_1 =\frac1{\sqrt{3}}, \qquad
		\alpha_2=-\sqrt{\frac32},\qquad
		\alpha_3=- \frac{1}{\sqrt{2}},
	 \]
   and
		\[
		d=v_1,\qquad
		e=\frac{1}{\sqrt{3}}(0,-1,-1,2)^T,\qquad
		f=\sqrt{\frac{2}{3}}(0,1,1,1)^T.
		\]
		Thus we obtain the completion
		\begin{multline*}
		E_3\left(\frac{1}{\sqrt{3}},-\sqrt{\frac32},-\frac{1}{\sqrt{2}};\ell_3-\ell_2,\frac{1}{\sqrt{3}}(-\ell_2-\ell_3+2\ell_4),\sqrt{\frac23}(\ell_2+\ell_3+\ell_4) \right)\\
		\sim
		\sgn(\ell_2)\sgn(\ell_3-\ell_2)\sgn(\ell_4-\ell_2).
		\end{multline*}

		We now turn to the case $\sgn(\ell_1)\sgn(\ell_2)\sgn(\ell_3-\ell_2)$ and set
		\[
		v_1:=(1,0,0,0)^T,\quad v_2:=(0,-1,1,0)^T,\quad v_3:=(0,1,0,0)^T.
		\]
		We compute
		\[
		||v_1||^2=||v_2||^2=-2,\quad ||v_3||^2=-1,\quad\langle v_1,v_2\rangle=0,\quad\langle v_1,v_3\rangle=\langle v_2,v_3\rangle=1.
		\]
		We plug these into the (\ref{la:explicitVectors}) to obtain
		\begin{align*}
		\lambda=1,\qquad
		\mu =1,\qquad
		\nu =1, \qquad
		\alpha_1 =0, \qquad
		\alpha_2 =-\frac{1}{2},\qquad
		\alpha_3=-\frac{1}{2},
		\end{align*}
		\begin{align*}
		d=v_1,\qquad
		e=v_2,\qquad
		f=\frac12\left(1,1,1,0\right)^T.
		\end{align*}
		We use the second part of Proposition \ref{th:Completions} to obtain the differential equation for $\mathcal{E}_3$ and then use Lemma \ref{la:CuspLimit} to obtain explicitely
		\begin{align*}
		-E\left( \ell_1\right)
		&- E\left(  \ell_3-\ell_2\right)
		-  E\left(\sqrt{2} \ell_2\right)
		\\& +\sgn\left(\ell_1+\ell_2+\ell_3\right)\left(E_2\left(0;\ell_1,\ell_3-\ell_2\right)
		-E_2\left(1; \ell_1, -\ell_1-2\ell_2 \right)
		+E_2\left(-1;  \ell_3-\ell_2, \ell_2+\ell_3\right)+1\right).
		\end{align*}
		In the same way, replacing $\ell_3$ by $\ell_4$ yields the completion
		for $\sgn(\ell_1), \sgn(\ell_2), \sgn(\ell_4-\ell_2)$
		\begin{align*}
		&\quad -E\left( \ell_1\right)
		- E\left(  \ell_4-\ell_2\right)
		-  E\left(\sqrt{2} \ell_2\right)
		\\&\quad  +\sgn(\ell_1+\ell_2+\ell_4)\Big(E_2\left(0;\ell_1,\ell_4-\ell_2\right)
		-E_2\left(1; \ell_1, -\ell_1-2\ell_2 \right)
		+	E_2\left(-1;  \ell_4-\ell_2, \ell_2+\ell_4\right)+1\Big).
		\end{align*}
	Combining all terms then gives the claim.

\end{proof}

\begin{corollary}
The function $F_1$ has a modular completion.
\end{corollary}
\begin{proof}
Define
\begin{align*}
	&\quad F(z_1,z_2,z_3,z_4;\tau)\\
	:&=\left(\sum_{n_2,n_3,n_4\ge0\atop{n_1>0\atop{n_2<\min\left(n_3,n_4\right)}}}-\sum_{n_2,n_3,n_4<0\atop{n_1\le0\atop{n_2\ge\max\left(n_3,n_4\right)}}}\right) (-1)^{n_1} q^{\frac{n_1^2}{2}+n_1n_2+n_1n_3+n_1n_4+n_2n_3+n_2n_4+n_3n_4} e^{2\pi i B(z,n)}.
\end{align*}
Noting $e^{2\pi i B(z,n)}=\zeta_1^{n_1+n_2+n_3+n_4}\zeta_2^{n_1+n_3+n_4}\zeta_3^{n_1+n_2+n_4}\zeta_4^{n_1+n_2+n_3}$, we obtain
\begin{align*}
F_1(\tau)&=
\lim_{w\rightarrow 0^+}
\left[\left(3+\frac{\partial}{2\pi i \partial w}\right)F\left(-w^2\tau,2w\tau+w^2\tau+\frac{\tau}{2},2w\tau+\frac{\tau}{2},2w\tau+\frac{\tau}{2}; \tau\right)\right]
\end{align*}
With $z(w):=\tau (-w^2,2w+w^2+1/2,2w+1/2,2w+1/2)$ we write for $w>0$ small enough
\begin{align*}
& F(z(w); \tau)
=\sum_{n\in \Z^4}\frac18 p\left(n+\frac{\im(z(w))}{v}
\right)(-1)^{n_1} q^{\frac{n_1^2}{2}+n_1n_2+n_1n_3+n_1n_4+n_2n_3+n_2n_4+n_3n_4}e^{2\pi i B(z(w),n)},
\end{align*}
since 
\begin{align*}
 p\left(n+\left(-w^2,2w+w^2+\frac12,2w+\frac12,2w+\frac12\right)^T\right)
=\begin{cases}
8\quad &\text{if } n_1>0,n_2\geq 0,n_3>n_2,n_4>n_2,\\
-8\quad &\text{if } n_1\leq 0,n_2< 0,n_3\leq n_2,n_4\leq n_2,\\
0 \quad &\text{otherwise}.
\end{cases}
\end{align*}

Since $p(x)=p(\sqrt{v}x)$, we obtain
\begin{align*}
&\quad 8F\left(z(w);\tau\right)
=\sum_{n\in \Z^4} p\left(n+\frac{\im(z(w))}{v}\right)(-1)^{n_1} q^{\frac12 n^TAn} e^{2\pi i B(z,n)}\\
&=\sum_{n\in \Z^4} p\left(n+\frac{\im(z(w))}{v}\right) q^{\frac12 n^TAn} e^{2\pi i B\left(z+A^{-1}\left(\frac12,0,0,0\right),n\right)}
=
\Theta_{0}\left(z(w)+\frac12 A^{-1}(1,0,0,0)^T;\tau\right).
\end{align*}
By Lemma \ref{lemma2.2} and Proposition \ref{prop:VigApprox},  $\widehat{\Theta}_0=\Theta_{0,\Z^4,A,\widehat{p},0}$ is the modular completion of $\Theta_0=\Theta_{0,\Z^4,A,{p},0}$. Note that for $w>0$ small enough, we have $z(w)\in\{(z,\tau)\in \C^n\times \C:B(c_j, \frac{y}{v})\not\in\Z \ (j\in\{0,1,2,3\})\}$ such that with Proposition \ref{prop:convergence} and Proposition \ref{prop:VigApprox}, 
\begin{align*}
\frac18 \lim_{w\rightarrow 0^+}
\left[\left(3+\frac{\partial}{2\pi i \partial w}\right)\widehat{\Theta}_{0}\left(z(w)+\frac12 A^{-1}(1,0,0,0)^T;\tau\right)\right]
\end{align*}
is the modular completion of
\begin{align*}
F_1(\tau)=\lim_{w\rightarrow 0^+}
\left[\left(3+\frac{\partial}{2\pi i \partial w}\right)F(z(w); \tau)\right]
=\frac18 \lim_{w\rightarrow 0^+}
\left[\left(3+\frac{\partial}{2\pi i \partial w}\right)\Theta_{0}\left(z(w)+\frac12 A^{-1}(1,0,0,0)^T;\tau\right)\right],
\end{align*}
which proves the claim.

\end{proof}


\begin{thebibliography}{99}
	\bibitem{ABMP} S. Alexandrov, S. Banerjee, J. Manschot, and B. Pioline, {\it Indefinite theta series and generalized error functions}, submitted for publication
	\bibitem{BRZ} K. Bringmann, L. Rolen, and S. Zwegers, {\it On the modularity of certain functions from Gromov-Witten theory of elliptic orbifolds}, Royal Society Open Science (2015), pages 150310.
	\bibitem{ChoHongKimLau} C. Cho, H. Hong, S. Kim, and S. Lau, {\it Lagrangian Floer potential oforbifold spheres}, preprint (2014), arXiv:1403.0990.
	\bibitem{ChoHongLau} C. Cho, H. Hong, and S. Lau, {\it Localized mirror functor for Lagrangian immersions, and homological mirror symmetry for $\mathbb{P}^1_{a,b,c}$,} J. Differential Geom. {\bf 106} (2017), no. 1, 45--126. 
	\bibitem{ChoHongLau2} C. Cho, H. Hong, and S. Lau, {\it Noncommutative homological mirror functor}, preprint (2015), arXiv:1512.07128. 
	\bibitem{ChoHongLee} C. Cho, H. Hong, and S. Lee, {\it Examples of matrix factorizations from SYZ}, SIGMA Symmetry Integrability Geom. Methods Appl. {\bf 8} (2012), Paper 053, 24.
	\bibitem{EZ} M. Eichler and D. Zagier, {\it The theory of Jacobi forms}, Progress in Mathematics {\bf 55}, Birkh\"auser, 1985.
	\bibitem{JonasPhD} J. Kaszi\'an, {\it Indefinite theta functions and mock modular forms of higher depth}, in preparation.
	\bibitem{Kudla} S. Kudla, {\it Theta integrals and generalized error functions}, preprint.
	\bibitem{KudlaMillson} S. Kudla and J. Millson, {\it The theta correspondence and harmonic forms I}, Math. Annalen {\bf 274} (1986), 353--378.
	\bibitem{LauZhou} S. Lau and J. Zhou, {\it Modularity of open Gromov-Witten potentials of elliptic orbifolds}, to appear in Communications in Number Theory and Physics.
	\bibitem{Vi} M. Vign\'eras, {\it S\'eries theta des formes quadratiques ind\'efinies} in: Modular functions of one variable VI, Springer lecture notes {\bf 627} (1977), 227-239.
	
	\bibitem{WR} M. Westerholt-Raum, {\it Indefinite Theta Series on Tetrahedral Cones}, preprint.
	\bibitem{Zwmult} S. Zwegers, {\it Multivariable Appell functions} (2010), preprint.
	\bibitem{Zw} S. Zwegers, {\it Mock theta functions}, Ph.D. Thesis, Universiteit Utrecht, (2002).
\end{thebibliography}
\end{document}